\documentclass[11pt,reqno]{amsart}
\usepackage[utf8]{inputenc}
\usepackage[T1]{fontenc}
\usepackage{lmodern}
\usepackage{amsmath}
\usepackage{amssymb}
\usepackage{amsfonts}
\usepackage{mathrsfs}
\usepackage{amsthm}
\usepackage{commath}
\usepackage{dsfont}
\usepackage{hyperref}
\usepackage{doi}
\usepackage{xcolor}

\usepackage{mathtools}

\newcommand{\R}{\mathbb{R}}
\newcommand{\Z}{\mathbb{Z}}

\newcommand{\atilde}{\widetilde{\alpha}}
\newcommand{\DG}{\mathcal{D}_{\Gamma}}
\newcommand{\DGe}{\mathcal{D}^e_{\Gamma}}
\newcommand{\EDG}{\widetilde{\D}_{\Gamma}}
\newcommand{\D}{\mathcal{D}}
\newcommand{\W}{\mathcal{W}}
\newcommand{\T}{\mathcal{\Tree}}
\newcommand{\Se}{\Stop^e}
\newcommand{\HD}{\mathsf{HD}}
\newcommand{\LD}{\mathsf{LD}}

\newcommand{\Stop}{\mathsf{Stop}}
\newcommand{\Tree}{\mathsf{Tree}}

\newcommand{\one}{\mathds{1}}
\DeclareMathOperator{\supp}{supp}

\newcommand\restr[2]{{% we make the whole thing an ordinary symbol
		\left.\kern-\nulldelimiterspace % automatically resize the bar with \right
		#1 % the function
		\right|_{#2} % this is the delimiter
}}

\newcommand{\HG}{\Hn{\Gamma}}
\newcommand{\Hn}[1]{\restr{\H^n}{#1}}
\renewcommand{\H}{\mathscr{H}}

\newcommand{\sg}{\sigma}

\DeclareMathOperator{\diam}{diam}
\DeclareMathOperator{\dist}{dist}
\DeclareMathOperator{\lip}{\text{Lip}}

\newtheorem{theorem}{Theorem}[section]
\newtheorem{lemma}[theorem]{Lemma}
\newtheorem{cor}[theorem]{Corollary}
\theoremstyle{definition}
\newtheorem{remark}[theorem]{Remark}

\numberwithin{equation}{section}

\newcounter{AbsConstants}

\title{Necessary condition for rectifiability involving Wasserstein distance $W_2$}

\author{Damian D\k{a}browski}
\address{Damian D\k{a}browski\newline\indent Departament de Matem\`atiques, Universitat Aut\`onoma de Barcelona; Barcelona Graduate School of Mathematics (BGSMath)\newline\indent Edifici C Facultat de Ci\`encies, 08193 Bellaterra (Barcelona, Catalonia) }
\email{ddabrowski@mat.uab.cat}

\begin{document}
	\begin{abstract}
	A Radon measure $\mu$ is $n$-rectifiable if $\mu\ll\H^n$ and $\mu$-almost all of $\supp\mu$ can be covered by Lipschitz images of $\R^n$. In this paper we give a necessary condition for rectifiability in terms of the so-called $\alpha_2$ numbers -- coefficients quantifying flatness using Wasserstein distance $W_2$. In a recent article we showed that the same condition is also sufficient for rectifiability, and so we get a new characterization of rectifiable measures.
\end{abstract}	
	\maketitle
	
	\section{Introduction}
	Let $1\le n\le d$ be integers. We say that a Radon measure $\mu$ on $\R^d$ is $n$-rectifiable if there exist countably many Lipschitz maps $f_i:\R^n\rightarrow\R^d$ such that 
	\begin{equation}\label{eq:measure concentrated on images}
	\mu(\R^d\setminus\bigcup_i f_i(\R^n))=0,
	\end{equation}
	and moreover $\mu$ is absolutely continuous with respect to $n$-dimensional Hausdorff measure $\H^n$. A set $E\subset \R^d$ is $n$-rectifiable if the measure $\Hn{E}$ is $n$-rectifiable. We will often omit $n$ and just write ``rectifiable''.
	
	The study of rectifiable sets and measures lies at the very heart of geometric measure theory. We refer the reader to \cite[Chapters 15--18]{mattila1999geometry} for some classical characterizations of rectifiability involving densities, tangent measures, and projections. The aim of this paper is to prove a necessary condition for rectifiability involving the so-called $\alpha_2$ coefficients. 
	
	\subsection{\texorpdfstring{$\alpha_p$ numbers}{alpha p numbers}} Coefficients $\alpha_p$ were introduced by Tolsa in \cite{tolsa2012mass}. In order to define them, we recall the definition of Wasserstein distance.
	
	Let $1\le p <\infty$, and let $\mu, \nu$ be two probability Borel measures on $\R^d$ satisfying $\int |x|^p\ d\mu<\infty,\ \int |x|^p\ d\nu<\infty$. The Wasserstein distance $W_p$ between $\mu$ and $\nu$ is defined as 
	\begin{equation*}
	W_p(\mu,\nu)=\left(\inf_{\pi}\int_{\R^d\times\R^d} |x-y|^p\ d\pi(x,y)\right)^{1/p},
	\end{equation*}
	where the infimum is taken over all transport plans between $\mu$ and $\nu$, i.e. Borel probability measures $\pi$ on $\R^d\times\R^d$ satisfying $\pi(A\times\R^d)=\mu(A)$ and $\pi(\R^d\times A)=\nu(A)$ for all measurable $A\subset\R^d$. The same definition makes sense if instead of probability measures we consider $\mu,\ \nu,\ \pi$ of the same total mass.
	
	Wasserstein distances are a way to measure the cost of {transporting} one measure to another, and they are of fundamental importance to the theory of optimal transport. For more information see for example \cite[Chapter 7]{villani2003topics} or \cite[Chapter 6]{villani2008optimal}.
	
	The idea behind $\alpha_p$ numbers is to quantify how far a given measure is from being a flat measure, that is, from being of the form $c\Hn{L}$ for some constant $c>0$ and some $n$-plane $L$. In order to measure it locally (say, in a ball $B$), we introduce the following auxiliary function. 
	
	Let $\varphi:\R^d\rightarrow [0,1]$ be a radial Lipschitz function satisfying $\varphi \equiv 1$ in $B(0,2)$, $\supp\varphi \subset B(0,3)$, and for all $x\in B(0,3)$
	\begin{gather*}
	c^{-1}\dist(x,\partial B(0,3))^2\le \varphi(x)\le c\dist(x,\partial B(0,3))^2,\\
	|\nabla\varphi(x)|\le c \dist(x,\partial B(0,3)),
	\end{gather*}
	for some constant $c>0$. For example, one could take $\varphi(x) = \phi(|x|)$ where $\phi:[0,\infty)\to [0,1]$ is such that $\phi(r)=1$ for $0\le r \le 2$, $\phi(r)=0$ for $r \ge 3$, and $\phi(r)=(3-r)^2$ for $2<r<3$. Given a ball $B = B(x,r)\subset\R^d$ we set
	\begin{equation}\label{eq:def of varphi B}
	\varphi_{{B}}(y) = \varphi\left(\frac{y-x}{r}\right).
	\end{equation}
	$\varphi_{{B}}$ can be seen as a regularized characteristic function of $B$. 
	
	For $1\le p <\infty$, a Radon measure $\mu$ on $\R^d$, a ball $B=B(x,r)\subset\R^d$ with $\mu(B)>0$, and an $n$-plane $L$ intersecting $B$, we define
	\begin{equation}\label{eq:definition alphas}
	\alpha_{\mu,p,L}(B) = \frac{1}{r\,\mu(B)^{1/p}}W_p(\varphi_B\mu,a_{B,L}\varphi_B\restr{\H^n}{L}),
	\end{equation}
	where $a_{{B,L}} = (\int \varphi_B\ d\mu)/(\int \varphi_B\ d\restr{\H^n}{L})$. We will usually omit the subscripts and just write $a$. We define also
	\begin{equation*}
	\alpha_{\mu,p}(B) = \inf_L \alpha_{\mu,p,L}(B),
	\end{equation*}
	where the infimum is taken over all $n$-planes $L$ intersecting $B$. For a ball $B=B(x,r)$ we will sometimes write $\alpha_{\mu,p}(x,r)$ instead of $\alpha_{\mu,p}(B)$, and we will do the same with all the other coefficients introduced below.
	
	Coefficients $\alpha_p$ were first defined in \cite{tolsa2012mass} with the aim of characterizing \emph{uniformly} rectifiable measures. The notion of uniform rectifiability, which can be seen as a more quantitative counterpart of rectifiability, was introduced by David and Semmes in \cite{david1991singular,david1993analysis}. We say that a measure $\mu$ is uniformly $n$-rectifiable if:
	\begin{itemize}
		\item[(i)] it is $n$-AD-regular, i.e. there exists a constant $C$ such that for all $x\in\supp\mu$ and $0<r<\diam(\supp\mu)$ we have $C^{-1}r^n\le\mu(B(x,r))\le Cr^n$,
		\item[(ii)] it has \emph{big pieces of Lipschitz images}, i.e. there exist constants $\theta, L>0$ such that for any $x\in\supp\mu$ and $0<r<\diam(\supp\mu)$ we may find an $L$-Lipschitz mapping $g$ from the $n$-dimensional ball $B^n(0,r)\subset\R^n$ into $\R^d$ satisfying
		\begin{equation*}
		\mu\big(B(x,r)\cap g(B^n(0,r))\big)\ge \theta r^n.
		\end{equation*}
	\end{itemize}
	A trivial example of a uniformly rectifiable measure is the surface measure on a Lipschitz graph.
	
	 In \cite{tolsa2012mass} Tolsa showed the following characterization of uniformly rectifiable measures:
	\begin{theorem}[{\cite[Theorem 1.2]{tolsa2012mass}}]\label{thm:characterization UR}
		Let $1\le p\le 2$. Suppose $\mu$ is an $n$-AD-regular measure on $\R^d$. Then, $\mu$ is uniformly rectifiable if and only if there exists  $C>0$ such that for any ball $B=B(z,R)$ centered at $\supp\mu$ we have
		\begin{equation*}
		\int_0^R\int_B \alpha_{\mu,p}(x,r)^2\ d\mu(x) \frac{dr}{r}\le CR^n.
		\end{equation*}	
	\end{theorem}
 
	In this paper we prove a necessary condition for rectifiability of measures which is of similar spirit.
	\begin{theorem}\label{thm:necessary condition}
		Let $\mu$ be an $n$-rectifiable measure on $\R^d$. Then for $\mu$-a.e. $x\in\R^d$
		\begin{equation}\label{eq:alpha2 square function}
		\int_0^{1}\alpha_{\mu,2}(x,r)^2\ \frac{dr}{r}<\infty.
		\end{equation}
	\end{theorem}
	In \cite[Theorem 1.3]{dabrowski2019sufficient} we show that \eqref{eq:alpha2 square function} is also a sufficient condition for rectifiability (we use a slightly different version of $\alpha_2$, but it does not matter, see \remref{rem:different alphas}). Putting the two results together, we get the following characterization.
	\begin{cor}\label{thm:thm equivalence}
		Let $\mu$ be a Radon measure on $\R^d$. Then $\mu$ is $n$-rectifiable if and only if for $\mu$-a.e. $x\in\R^d$ we have
		\begin{equation*}
		\int_0^{1}\alpha_{\mu,2}(x,r)^2\ \frac{dr}{r}<\infty.
		\end{equation*}
	\end{cor}
	
	\begin{remark}\label{rem:sharpness}
		The characterization above is sharp in the following sense. Suppose $1\le p \le q<\infty$. Then it follows easily by H\"{o}lder's inequality, definition of $\alpha_p$ numbers, and the fact that $\supp\varphi_B\subset 3B$, that
		\begin{equation*}
		\alpha_{\mu,p}(B)\le\bigg(\frac{\mu(3B)}{\mu(B)}\bigg)^{1/p-1/q}\alpha_{\mu,q}(B).
		\end{equation*}
		Hence, for doubling measures, $\alpha_p$ numbers are increasing in $p$. It is well known that rectifiable measures are pointwise doubling, i.e. \begin{equation}\label{eq:pointwise doubling}
		\limsup_{r\to 0^+}\frac{\mu(B(x,2r))}{\mu(B(x,r))}<\infty\qquad\text{for $\mu$-a.e. $x\in\R^d$,}
		\end{equation}
		and so the finiteness of $\alpha_2$ square function \eqref{eq:alpha2 square function} implies finiteness of $\alpha_p$ square function for any $1\le p \le 2.$ However, in general one cannot expect finiteness of $\alpha_p$ square function for $p>2$, see \remref{rem:sharpness 2}. In other words, \thmref{thm:necessary condition} cannot be improved.
	\end{remark}
	\begin{remark}\label{rem:different alphas}
		For technical reasons, in \cite{dabrowski2019sufficient} we define $\alpha_p$ numbers normalizing by $\mu(3B)$ (i.e. in \eqref{eq:definition alphas} we replace $\mu(B)$ with $\mu(3B)$). Of course, the $3B$-normalized coefficients are \emph{smaller} than the $B$-normalized variant used here. Hence, if \eqref{eq:alpha2 square function} is finite for $B$-normalized $\alpha_2$ numbers, then it is finite for $3B$-normalized $\alpha_2$ numbers, and so \cite[Theorem 1.3]{dabrowski2019sufficient} may be applied to get Corollary \ref{thm:thm equivalence}.
	\end{remark}
	
	It is worthwhile to compare this result with other recent characterizations of rectifiability which all involve some sort of scale-invariant quantities measuring flatness. 
	\subsection{\texorpdfstring{$\beta_p$ numbers}{beta p numbers}} The first flatness-quantifying coefficients to be defined were Jones' $\beta$ numbers, originating in \cite{jones1990rectifiable,david1991singular,david1993analysis}. For $1\le p <\infty$ and a Radon measure $\mu$ on $\R^d$ set
	\begin{equation}\label{eq:beta definition}
	\beta_{\mu,p}(x,r)=\inf_L\left(\frac{1}{r^n}\int_{B(x,r)}\left(\frac{\dist(y,L)}{r}\right)^p\ d\mu(y)\right)^{1/p},
	\end{equation}
	where the infimum runs over all $n$-planes $L$ intersecting $B(x,r)$. 
%	The letter $h$ in the superscript stands for \emph{homogeneous} and refers to the normalizing factor $r^{n}$. For us it will be more convenient to normalize by $\mu(B(x,r))$ instead, and so we define
%	\begin{equation*}
%	\beta_{\mu,p}(x,r)=\inf_L\left(\frac{1}{\mu(B(x,r))}\int_{B(x,r)}\left(\frac{\dist(y,L)}{r}\right)^p\ d\mu(y)\right)^{1/p}.
%	\end{equation*}
	Let us also define upper and lower $n$-dimensional densities of a Radon measure $\mu$ at $x\in\R^d$ as
	\begin{equation*}
	\Theta^{n,*}(x,\mu)=\limsup_{r\to 0^+}\frac{\mu(B(x,r))}{r^n},\qquad \Theta_{*}^n(x,\mu)=\liminf_{r\to 0^+}\frac{\mu(B(x,r))}{r^n},
	\end{equation*}
	respectively. If both quantities are equal, we set $\Theta^{n}(x,\mu) = \Theta^{n,*}(x,\mu) = \Theta_{*}^n(x,\mu)$ and we call it $n$-dimensional density.
	
	In \cite{tolsa2015characterization} it was shown that for a rectifiable measure $\mu$ we have
	\begin{equation}\label{eq:Jones square function finite}
	\int_0^1 \beta_{\mu,2}(x,r)^2\ \frac{dr}{r}<\infty\qquad\text{for $\mu$-a.e. $x\in\R^d.$}
	\end{equation}
	%Since for rectifiable measures $\Theta^{n}(x,\mu)$ exists and is positive and finite almost everywhere, it is easy to see that \eqref{eq:Jones square function finite} is true also after replacing $\beta_{\mu,2}^h(x,r)$ by $\beta_{\mu,2}(x,r)$.
	
	On the other hand, Azzam and Tolsa proved in \cite{azzam2015characterization} that if a Radon measure $\mu$ satisfies \eqref{eq:Jones square function finite} and 
	\begin{equation}\label{eq:upper density bdd}
	0<\Theta^{n,*}(x,\mu)<\infty\qquad\text{for $\mu$-a.e. $x\in\R^d$,}
	\end{equation}
	then $\mu$ is $n$-rectifiable. More recently, Edelen, Naber and Valtorta \cite{edelen2016quantitative} managed to weaken the assumption \eqref{eq:upper density bdd} to
	\begin{equation}\label{eq:weakened upper density bdd}
	\Theta^{n,*}(x,\mu)>0\quad \text{and} \quad \Theta_*^{n}(x,\mu)<\infty\qquad\text{for $\mu$-a.e. $x\in\R^d$.}
	\end{equation}
	An alternative proof showing that \eqref{eq:Jones square function finite} and \eqref{eq:weakened upper density bdd} are sufficient for rectifiability is given in \cite{tolsa2017rectifiability}.
	\begin{theorem}[\cite{tolsa2015characterization,azzam2015characterization,edelen2016quantitative}]\label{thm:beta characterization}
	Let $\mu$ be a Radon measure on $\R^d$. Then, $\mu$ is $n$-rectifiable if and only if \eqref{eq:Jones square function finite} and \eqref{eq:weakened upper density bdd} hold for $\mu$-a.e. $x\in\R^d$.
	\end{theorem}
	Contrary to Corollary \ref{thm:thm equivalence}, some sort of assumptions on densities of measure seem to be unavoidable because $\beta_2$ numbers are ``weaker'' than $\alpha_2$ numbers (see \lemref{lem:alpha2 controls alpha and beta2}). What we mean by that is the following: coefficients $\beta_{p}$ measure how close is $\supp\mu$ to being contained in an $n$-plane, and so they do not see holes or high concentrations of measure. \emph{Any} measure with support contained in an $n$-plane will have all $\beta$ numbers equal to 0 -- even Dirac mass! Moreover, due to the normalizing factor $r^{n}$ in \eqref{eq:beta definition}, $\beta$ numbers do not charge higher dimensional measures properly (note that the $(n+1)$-dimensional Lebesgue measure satisfies \eqref{eq:Jones square function finite}). Coefficients $\alpha_p$, on the other hand, penalize such phenomena.
	
	The choice of $p=2$ in the above considerations is not arbitrary. Condition \eqref{eq:Jones square function finite} with $\beta_{\mu,2}(x,r)$ replaced by $\beta_{\mu,p}(x,r)$ is necessary for rectifiability only for $1\le p\le 2$. On the other hand, \eqref{eq:Jones square function finite} together with \eqref{eq:upper density bdd} imply rectifiability only for $p\ge 2$. See \cite{tolsa2017rectifiability} for relevant counterexamples. Still, if instead of \eqref{eq:upper density bdd} we assume that $\Theta_*^{n}(\mu,x)>0$ and $\Theta^{n,*}(\mu,x)<\infty$  for $\mu$-a.e. $x\in\R^d$, then the finiteness of $\beta_p$ square function for certain $p<2$ becomes sufficient for rectifiability, see \cite{pajot1997conditions,badger2016two}.
	
	\begin{remark}\label{rem:sharpness 2}
		The example from \cite{tolsa2017rectifiability} shows that one cannot expect finiteness of the $\alpha_p$ square function when $p>2$. Indeed, it is easy to see that $\alpha_p$ numbers bound from above $\beta_p$ numbers (see \lemref{lem:alpha2 controls alpha and beta2}, the same proof works with arbitrary $1\le p<\infty$). Tolsa gave an example of a rectifiable measure such that for all $p>2$ the square function involving $\beta_p$ in infinite almost everywhere. Hence, the $\alpha_p$ square function of that measure is also infinite almost everywhere.
	\end{remark}
	
	Let us mention that modified versions of $\beta$ numbers are also used to study a competing notion of rectifiability for measures, the so-called \emph{Federer rectifiability}. We say that a measure is $n$-rectifiable in the sense of Federer if it satisfies \eqref{eq:measure concentrated on images}, and no absolute continuity with respect to $\H^n$ is required. Dropping the absolute continuity assumption makes such measures very difficult to characterize. A surprising example of a doubling, Federer $1$-rectifiable measure supported on the whole plane was found by Garnett, Killip and Schul \cite{garnett2010doubling}. Nevertheless, for $n=1$ significant progress has been achieved in \cite{lerman2003quantifying,badger2015multiscale,badger2016two,azzam2016characterization,badger2017multiscale,martikainen2018boundedness}. See also a recent survey of Badger \cite{badger2018generalized}.
	
	\thmref{thm:necessary condition} yields an easy corollary involving \emph{bilateral} $\beta$ numbers. Set
	\begin{multline*}
	b\beta_{\mu,2}(x,r)^2=\\
	\inf_L\frac{1}{r^n}\left(\int_{B(x,r)}\left(\frac{\dist(y,L)}{r}\right)^2\ d\mu(y) + \int_{B(x,r)}\left(\frac{\dist(y,\supp\mu)}{r}\right)^2\ d\Hn{L}(y)\right).
	\end{multline*}
	As shown in \lemref{lem:alpha2 controls alpha and beta2}, if a ball $B(x,r)$ satisfies $\mu(B(x,r))\approx r^n$ (see Subsection \ref{subsec:notation} for the precise meaning of $\approx$ symbol), then coefficients $\alpha_{\mu,2}(x,r)$  bound from above $b\beta_{\mu,2}(x,r)$. Since for $n$-rectifiable measure $\mu$  we have $0<\Theta^{n}(\mu,x)<\infty$ $\mu$-almost everywhere, we immediately get the following.
	\begin{cor}
		Let $\mu$ be an $n$-rectifiable measure on $\R^d$. Then for $\mu$-a.e. $x\in\R^d$ we have
		\begin{equation*}
		\int_0^{1}b\beta_{\mu,2}(x,r)^2\ \frac{dr}{r}<\infty.
		\end{equation*}
	\end{cor}
	
	\subsection{\texorpdfstring{$\alpha$}{alpha} numbers} Another kind of coefficients quantifying flatness that has attracted a lot of interest are $\alpha$ numbers, first introduced in \cite{tolsa2008uniform}. Their definition is very similar to that of $\alpha_p$ coefficients, and in fact they can be seen as a variant of $\alpha_1$ numbers, see \cite[Section 5]{tolsa2012mass}. 
	
	Like before, we define a distance on the space of Radon measures. Given Radon measures $\mu, \nu$, and an open ball $B$ we set
	\begin{equation*}
	F_B(\mu, \nu)=\sup\left\{\left\lvert\int\phi\ d\mu - \int\phi\ d\nu \right\rvert\ :\ \phi\in\text{Lip}_1(B)\right\},
	\end{equation*}
	where
	\begin{equation*}
	\lip_1(B)=\{\phi\ :\ \lip(\phi)\le 1,\ \supp \phi\subset B \}.
	\end{equation*}	
	The coefficient $\alpha$ of a measure $\mu$ in a ball $B=B(x,r)$ is defined as
	\begin{equation*}
	\alpha_{\mu}(B)=\inf_{c,L}\frac{1}{r\, \mu(B)} F_B(\mu,c\restr{\H^n}{L}),
	\end{equation*}
	where the infimum is taken over all $n$-planes $L$ and all $c\ge 0$ (we do not demand a priori that $\mu(B)=c\Hn{L}(B)$).
	
	Tolsa showed in \cite{tolsa2015characterization} that given a rectifiable measure $\mu$ we have
	\begin{equation}\label{eq:alpha square function}
	\int_0^1 \alpha_{\mu}(x,r)^2\ \frac{dr}{r}<\infty\qquad\text{for $\mu$-a.e. $x\in\R^d.$}
	\end{equation}
	One might ask if \eqref{eq:alpha square function} is also a sufficient condition for rectifiability. Partial answers to that question were given in \cite{azzam2016wasserstein} and \cite{orponen2018absolute}. Very recently Azzam, Tolsa and Toro \cite{azzam2018characterization} proved that a measure $\mu$ satisfying \eqref{eq:alpha square function} which is also pointwise doubling, i.e. such that \eqref{eq:pointwise doubling} holds, is rectifiable. Since rectifiable measures satisfy \eqref{eq:pointwise doubling}, the following characterization holds.
	\begin{theorem}[{\cite{tolsa2015characterization,azzam2018characterization}}]\label{thm:alpha doubling characterization}
		Let $\mu$ be a Radon measure on $\R^d$. Then $\mu$ is $n$-rectifiable if and only if \eqref{eq:alpha square function} and \eqref{eq:pointwise doubling} hold for $\mu$-a.e. $x\in\R^d$.
	\end{theorem}
	In the same paper authors construct a purely unrectifiable measure satisfying \eqref{eq:alpha square function}, and so the pointwise doubling assumption \eqref{eq:pointwise doubling} cannot be omitted. Let us remark that in the characterization from Corollary \ref{thm:thm equivalence} we do not need to assume any doubling property.
	
	We mention briefly yet another kind of square functions used to describe rectifiability. \cite{tolsa2014rectifiability} and \cite{tolsa2017rectifiable} are devoted to the so-called $\Delta$ numbers, defined as $\Delta_{\mu}(x,r)=|\frac{\mu(B(x,r))}{r^n} - \frac{\mu(B(x,2r))}{(2r)^n}|$. The results from \cite{tolsa2014rectifiability} characterize rectifiable measures satisfying $0<\Theta^n_*(\mu,x)\le\Theta^{n,*}(\mu,x)<\infty$ for $\mu$-a.e $x\in\R^d$. In \cite{tolsa2017rectifiable} it was shown that for $n=1$ analogous results hold under the weaker assumption $0<\Theta^{1,*}(x,\mu)<\infty$ for $\mu$-a.e. $x\in\R^d$.
	
%	Since coefficients $\alpha_2$ bound from above simultaneously $\alpha$ numbers and $\beta_2$ numbers (see \lemref{lem:alpha2 controls alpha and beta2}), finiteness of the $\alpha_2$ square function \eqref{eq:alpha2 square function} implies finiteness of $\alpha$ and $\beta_2$ square functions \eqref{eq:alpha square function}, \eqref{eq:Jones square function finite}. As it turns out, the two latter conditions are already sufficient for rectifiability. 
%	
%	%	We define also
%	%	\begin{align*}
%	%	\beta_{\mu,1}(B)&=\inf_L\frac{1}{\mu(B)}\int_B\frac{\dist(y,L)}{r(B)}\ d\mu(y),\\
%	%	\beta_{\mu,2}(B)^2&=\inf_L\frac{1}{\mu(B)}\int_B\left(\frac{\dist(y,L)}{r(B)}\right)^2\ d\mu(y).
%	%	\end{align*}
%	%	Let us define
%	%	\begin{equation*}
%	%	\alpha_{\mu,2}(B)=\inf_{L}\frac{1}{\mu(B)^{1/2}r}W_2(\varphi_B\mu,c_B\varphi_B\H_L^n),
%	%	\end{equation*}
%	%	where 
%	%	\begin{equation*}
%	%	c_B = \frac{\int\varphi_B\ d\mu}{\int\varphi_B\ d\H^n_L}
%	%	\end{equation*}
%	%	Also, we denote by $L_B$ some $n$-plane for which the infimum above is attained.
%	\begin{theorem}\label{thm:main theorem}
%		Let $\mu$ be a Radon measure on $\R^d$. Then $\mu$ is $n$-rectifiable if and only if for $\mu$-a.e. $x\in\R^d$ we have
%		\begin{equation*}
%		\int_0^{1}\alpha_{\mu}(x,r)^2\ \frac{dr}{r}<\infty,
%		\end{equation*}
%		and
%		\begin{equation*}
%		\int_0^{1}\beta_{\mu,2}(x,r)^2\ \frac{dr}{r}<\infty.
%		\end{equation*}
%	\end{theorem}
	
	\subsection{Localizing \thmref{thm:necessary condition} and Organization of the Paper}\label{subsec:outline} 
	\autoref{thm:necessary condition} follows easily from the following lemma.
	\begin{lemma}\label{lem:general_with_Gamma}
		Let $\mu$ be an $n$-rectifiable measure on $\R^d$, and let $\Gamma\subset\R^d$ be an $n$-dimensional $1$-Lipschitz graph. Suppose $R\in \DG$ with $\ell(R)=1$ (see \eqref{eq:definition of DG} for the defintion of $\DG$). Then, for any $\varepsilon>0$, there exists a set $R'\subset R$ such that $\mu(R')\ge (1-\varepsilon)\mu(R)$ and
		\begin{equation}\label{eq:alpha square function on R'}
		\int_{R'}\int_0^{1}\alpha_{\mu,2}(x,r)^2\ \frac{dr}{r}\ d\mu(x)<\infty.
		\end{equation}
	\end{lemma}

	\begin{proof}[{Proof of \thmref{thm:necessary condition} using \lemref{lem:general_with_Gamma}}]
		Let $\mu$ be $n$-rectifiable. It is well known that if one replaces Lipschitz images in \eqref{eq:measure concentrated on images} by Lipschitz graphs, or $C^1$ manifolds, the definition of rectifiability remains unchanged (see e.g. \cite[Theorem 15.21]{mattila1999geometry}). Each $C^1$ manifold is contained in a countable union of (possibly rotated) Lipschitz graphs $\Gamma$ with $\lip(\Gamma)\le 1$. Hence, there exists a countable family of $n$-dimensional $1$-Lipschitz graphs $\Gamma_i$ such that
		\begin{equation*}
		\mu\big(\R^d\setminus \bigcup_i\Gamma_i\big)=0.
		\end{equation*}
		Each $\Gamma_i$ is a countable union of dyadic $\Gamma_i$-cubes $R_i^j\in \mathcal{D}_{\Gamma_i}$ satisfying $\ell(R_i^j)=1$. Clearly, $\mu(\R^d\setminus \bigcup_{i,j}R_i^j)=0$. 
		
		Now, denote the set of $x$ where \eqref{eq:alpha2 square function} does \emph{not} hold by $\mathcal{B}$, and suppose that $\mu(\mathcal{B})>0$. Then, there exists $R_i^j$ such that $\mu(\mathcal{B}\cap R_i^j)>0$. Let $\varepsilon>0$ be such that $\mu(\mathcal{B}\cap R_i^j)>2\varepsilon\mu(R_i^j)$. Applying \lemref{lem:general_with_Gamma} to $R_i^j$ and $\varepsilon$ as above we reach a contradiction. Thus, $\mu(\mathcal{B})=0$.
	\end{proof}
	The rest of the article is dedicated to proving \lemref{lem:general_with_Gamma}. Let us give a brief outline of the proof.
	
	We introduce the necessary tools in Section \ref{sec:dyadic}. In Section \ref{sec:estimates of aplha2} we show various estimates of $\alpha_2$ coefficients, usually relying heavily on the results from \cite{tolsa2012mass}. In Section \ref{sec:approximating measures nuQ} we define a family of measures $\{\nu_Q\}_{Q\in\DG}$, where $\nu_Q\ll \Hn{\Gamma}$, and each $\nu_Q$  approximates $\mu$ in some ball around $Q$. Roughly speaking, $\nu_Q$ is defined by projecting the measure of Whitney cubes onto the graph $\Gamma$ -- but only those Whitney cubes whose sidelength is not much bigger than $\ell(Q)$. Then, we construct a tree of good cubes satisfying
	\begin{equation*}
	\sum_{\substack{Q\in \T}}\alpha_{\nu_Q,2}(\widetilde{B}_Q)^2 \ell(Q)^n <\infty,
	\end{equation*}
	where $\widetilde{B}_Q$ are balls with the same center as the corresponding cube $Q$. The stopping region of the tree of good cubes is small. In Section \ref{sec:pass from approximating to mu} we use the estimate above to show that actually
	\begin{equation*}
	\sum_{\substack{Q\in \T}} \alpha_{\mu,2}(\widetilde{B}_Q)^2 \ell(Q)^n <\infty.
	\end{equation*}
	Using the inequality above, we prove \eqref{eq:alpha square function on R'} with $R'=R\setminus\bigcup_{Q\in\Stop(\Tree)}Q$. This finishes the proof of \lemref{lem:general_with_Gamma}.
	
	\subsection*{Acknowledgements} The author would like to thank Xavier Tolsa for all his help and guidance. He acknowledges the support of the Spanish Ministry of Economy and Competitiveness, through the María de Maeztu Programme for Units of Excellence in R\&D (MDM-2014-0445), and also partial support from 2017-SGR-0395 (Catalonia) and MTM-2016-77635-P (MINECO, Spain).
\section{Preliminaries}\label{sec:dyadic}

\subsection{Notation}\label{subsec:notation}
Throughout the paper we will write $A\lesssim B$ whenever $A\le CB$ for some constant $C$, the so-called ``implicit constant''. All such implicit constants may depend on dimensions $n, d$, and we will not track this dependence. If the implicit constant depends also on some other parameter $t$, we will write $A\lesssim_{t} B$. The notation $A\approx B$ means $A\lesssim B\lesssim A$, and $A\approx_t B$ means $A\lesssim_t B\lesssim_t A$. Moreover, if symbols $\lesssim$ or $\approx$ appear in the assumptions of a lemma, then the implicit constant of the proven estimate will depend on the implicit constants from the assumptions (see \lemref{lem:alpha2 controls alpha and beta2} for example).

We denote by $B(z,r)\subset\R^d$ an open ball with center at $z\in\R^d$ and radius $r>0$. Given a ball $B$, its center and radius are denoted by $z(B)$ and $r(B)$, respectively. If $\lambda>0$, then $\lambda B$ is defined as a ball centered at $z(B)$ of radius $\lambda r(B)$.

%	For a ball $B$, we define the $n$-dimensional density of $\mu$ at $B$ as
%	\begin{equation*}
%	\Theta_{\mu}(B) = \frac{\mu(B)}{r(B)^n}.
%	\end{equation*}

Given two $n$-planes $L_1, L_2$, let $L_1'$ and $L_2'$ be the respective parallel $n$-planes passing through $0$. Then,
\begin{equation*}
\measuredangle(L_1,L_2)=\dist_H(L_1'\cap B(0,1),\ L_2'\cap B(0,1)),
\end{equation*}
where $\dist_H$ stands for Hausdorff distance between two sets. $\measuredangle(L_1,L_2)$ can be seen as a ``sine of the angle between $L_1$ and $L_2$,'' and we always have $\measuredangle(L_1,L_2)\in [0,1]$.

Given an $n$-plane $L$, we will denote the orthogonal projection onto $L$ by $\Pi_L$. 
%The orthogonal projection onto $L^{\perp}$ will be denoted by $\Pi^{\perp}_L$.

For a Borel measure $\nu$ on $\R^d$ and a Borel map $T:\R^d\to \R^d$, we denote by $T_*\nu$ the pushforward of $\nu$, that is, a measure on $\R^d$ such that for all Borel $A\subset\R^d$
\begin{equation*}
T_*\nu(A) = \nu(T^{-1}(A)).
\end{equation*}

In expressions of the form $W_p(\mu_1,a\mu_2)$, the letter $a$ will always mean the unique constant for which the total mass of $a\mu_2$ is equal to that of $\mu_1$. In other words,
\begin{equation*}
a=\frac{\mu_1(\R^d)}{\mu_2(\R^d)}.
\end{equation*}
It may happen that $a$ appears in the same line several times, and every time refers to a different quantity. We hope that this will not cause too much confusion.

Let us once and for all fix measure $\mu$, an $n$-dimensional $1$-Lipschitz graph $\Gamma$, and a small constant $0<\varepsilon\ll 1$ for which we are proving \lemref{lem:general_with_Gamma}. We fix also a coordinate system such that $\Gamma = \{(x,A(x)) : x\in\R^{n}\}\subset \R^d,$ where $A:\R^n\rightarrow\R^{d-n}$ is a $1$-Lipschitz map.

We will denote by $L_0$ the subspace of $\R^d$ formed by the points whose last $d-n$ coordinates are zeros, so that $\Gamma$ is a graph over $L_0$. 
We will write $\Pi_0$ and $\Pi_{\Gamma}$ to denote projections onto $L_0$ and $\Gamma,$ respectively, orthogonal to $L_0$. 
%For any set $G\subset\R^d$ we set $\H^n_G$ as the restriction  of $\H^n$ to $G$.
For the sake of convenience, instead of dealing with the usual surface measure on $\Gamma$ we will work with
\begin{equation*}
	\sigma=(\Pi_{\Gamma})_*\Hn{L_0},
\end{equation*} which is comparable to $\HG$. 

Given a ball $B\subset\R^d$ centered at $\Gamma$ denote by $L_B$ an $n$-plane minimizing $\alpha_{\sigma,2}(B)$ (note that for an open ball $B$, it could happen that $L_B\cap B=\varnothing$).  Concerning the existence of minimizers, it follows easily from the fact that $W_2$ metrizes weak convergence of measures (see e.g. \cite[Theorem 6.9]{villani2008optimal}), from good compactness properties of weak convergence, and from the fact that the minimizing sequence is of the special form $\varphi_B a_{B,L_k}\Hn{L_k}$. There may be more than one minimizing plane; if that happens, we simply choose one of them. 

For any Radon measure $\nu$ such that $\nu(B)>0$ we set
\begin{equation*}
	\widehat{\alpha}_{\nu,2}(B) = \alpha_{\nu,2,L_B}(B).
\end{equation*}
Clearly, $\widehat{\alpha}_{\nu,2}(B)\ge\alpha_{\nu,2}(B)$. We will show that
\begin{equation}\label{eq:stronger alpha square function}
	\int_{R'}\int_0^{1}\widehat{\alpha}_{\mu,2}(x,r)^2\ \frac{dr}{r}\ d\mu(x)<\infty,
\end{equation}
which implies \eqref{eq:alpha square function on R'}.

\subsection{\texorpdfstring{$\Gamma$}{Gamma}-cubes}\label{subsec:Gamma cubes}
We denote by $\D_{\R^n}, \D_{\R^d}$ the dyadic lattices on $L_0$ and $\R^d$, respectively. We assume the cubes to be half open-closed, i.e. of the form
\begin{equation*}
Q=\bigg[\frac{k_1}{2^{j}}, \frac{k_1+1}{2^j} \bigg)\times \dots \times\bigg[\frac{k_i}{2^{j}}, \frac{k_i+1}{2^j} \bigg),
\end{equation*}
where $i=n$ for $\D_{\R^n}$, $i=d$ for $\D_{\R^d}$, and $k_1,\ \dots,\ k_i,\ j,$ are arbitrary integers. The sidelength of $Q$ as above will be denoted by $\ell(Q)=2^{-j}$.

The dyadic lattice on $\Gamma$ is defined as
\begin{equation}\label{eq:definition of DG}
\DG = \{\Pi_{\Gamma}(Q_0) : Q_0\in \mathcal{D}_{\R^n} \}.
\end{equation}
The elements of $\DG$ will be called $\Gamma$-cubes, or just cubes. For every $Q\in \DG$ and the corresponding $Q_0\in \mathcal{D}_{\R^n}$ we define the sidelength of $Q$ as $\ell(Q)= \ell(Q_0)$, and the center of $Q$ as $z_Q=\Pi_{\Gamma}(z_{Q_0}), $ where $z_{Q_0}$ is the center of $Q_0$. We set 
\begin{gather*}
B_Q= B(z_Q,3\diam(Q)),\\
\widetilde{B}_Q= \Lambda B_Q,
\end{gather*}
where $\Lambda=\Lambda(n)>1$ is a constant fixed during the proof. We define also
\begin{gather*}
\varphi_Q = \varphi_{B_Q},\\
L_Q=L_{B_Q},\\
%\widehat{\alpha}_{\mu,2}(Q) = \widehat{\alpha}_{\mu,2}(B_Q),\\
V(Q) = \{x\in\R^d: \Pi_{\Gamma}(x)\in Q\}.
\end{gather*}
Recall that $L_{B_Q}$ is the $n$-plane minimizing $\alpha_{\sigma,2}(B_Q)$, and that $\varphi_{B_Q}$ was defined in \eqref{eq:def of varphi B}. The ``$V$'' in $V(Q)$ stands for ``vertical'', since $V(Q)$ is a sort of vertical cube.
Note also that $Q\subset B_Q\subset \widetilde{B}_Q$ and $r(B_Q)\approx \ell(Q)$.

Given $P\in\DG$, we will write $\DG(P)$ to denote the family of $Q\in\DG$ such that $Q\subset P$.

\begin{remark}\label{remark:mu compactly supported}
	Let us fix $R\in \DG$ with $\ell(R)=1$ for which we are proving \lemref{lem:general_with_Gamma}. Note that for $x\in R$ and $0<r<1$ computing $\alpha_{\mu,2}(x,r)$ involves only $\restr{\mu}{B}$, where $B$ is some ball containing $R$. Thus, when proving \eqref{eq:stronger alpha square function}, we may and will assume that $\mu$ is a finite, compactly supported measure.
\end{remark}

%\subsection*{Translated lattices}
For every $e\in \{0,1\}^n$ consider the translated dyadic grid on $L_0$
\begin{equation*}
\mathcal{D}^e_{\R^n} = \frac{1}{3}(e,0\dots,0) + \mathcal{D}_{\R^n},
\end{equation*}
and the corresponding translated dyadic grid on $\Gamma$
\begin{equation*}
\DGe = \{\Pi_{\Gamma}(Q) : Q\in \mathcal{D}^e_{\R^n} \}.
\end{equation*}
Let us also define the translated dyadic lattice on $\R^d$
\begin{equation*}
\mathcal{D}^e_{\R^d} = \frac{1}{3}(e,0,\dots,0) + \mathcal{D}_{\R^d}.
\end{equation*}
The union of all translated dyadic grids on $\Gamma$ will be called an extended grid on $\Gamma$:
\begin{equation*}
\widetilde{\D}_{\Gamma} = \bigcup_{e\in\{0,1\}^n}\DGe.
\end{equation*}
For each $Q\in \EDG$ we define $B_Q,\ \varphi_Q$ etc. in the same way as for $Q\in\DG$. 

The main reason for introducing the extended grid is to use a variant of the well-known one-third trick, which was already used in this context by Okikiolu \cite{okikiolu1992characterization}.

\begin{lemma}\label{lem:Qtilde_corresp_to_Q}
	There exists $k_0=k_0(n,\Lambda)>0$ such that for every $Q\in\DG$ with $\ell(Q)\leq 2^{-k_0}$ there exists $P_Q\in \EDG$ satisfying $\ell(P_Q)=2^{k_0}\ell(Q)$ and $3\widetilde{B}_Q\subset V(P_Q)$.
\end{lemma}
\begin{proof}
	First, we remark that for every $j\geq 0$ and for every $x\in L_0$ there exists $e\in\{0,1\}^n$ and $P\in\D^e_{\R^n} $ with $\ell(P)=2^{-j}$ and $x\in\frac{2}{3}P$. For a nice proof of this fact see \cite[Section 3]{lerman2003quantifying}.
	
	Now, consider the point $\Pi_0(z_Q)$. If we take $P\in\D^e_{\R^n}$ with $\ell(P)=2^{k_0}\ell(Q)$ such that $\Pi_0(z_Q)\in\frac{2}{3}P$, we see that the $n$-dimensional ball $B^n(\Pi_0(z_Q),9\Lambda\diam(Q))$ is contained in $P$ as soon as $\frac{2^{k_0}}{3}\ell(Q) \ge 9\Lambda\diam(Q).$ 
	
	It follows that for $P_Q\in\DGe$ such that $\Pi_0(P_Q)=P$ we have $3\widetilde{B}_Q\subset V(P_Q)$.
\end{proof}

It may happen that the cube $P_Q\in\EDG$ from the lemma above is not unique, so let us just fix one for each $Q\in \DG$. The direction $e\in\{0,1\}^n$ such that $P_Q\in\DGe$ will be denoted by $e(Q)$, and the integer $k$ such that $\ell(P_Q)=2^{k_0}\ell(Q)=2^{-k}$ will be denoted by $k(Q)$.

We will use later on the fact that
\begin{equation}\label{eq:k0 big}
9\diam(Q)\le 2^{k_0}\ell(Q) = 2^{-k(Q)}.
\end{equation}

\subsection{Whitney cubes}
A very useful tool for approximating the measure $\mu$ close to $\Gamma$ are Whitney cubes. For each $e\in  \{0,1\}^n$ we consider the decomposition of $\R^d\setminus\Gamma$ into a family $\mathcal{W}^e$ of Whitney dyadic cubes from  $\D^e_{\R^d}$. That is, the elements of $\W^e\subset \D^e_{\R^d}$ are pairwise disjoint, their union equals $\R^d\setminus\Gamma$, and there exist dimensional constants $K>20, D_0\ge 1$ such that for every $Q\in\W^e$
\begin{enumerate}
	\item[a)] $10Q\subset\R^d\setminus\Gamma$,
	\item[b)] $KQ\cap\Gamma\ne \varnothing$,
	\item[c)] there are at most $D_0$ cubes $Q'\in\W^e$ such that $10Q\cap10Q'\ne \varnothing $. Furthermore, for such cubes $Q'$ we have $\ell(Q')\approx \ell(Q)$.
\end{enumerate}
For the proof see \cite[Chapter VI, \S 1]{stein1970singular} or \cite[Appendix J]{grafakos2008classical}. Moreover, it is not difficult to construct Whitney cubes in such a way that if $y\in\Gamma,\ Q\in\W^e$ and $B(y,r)\cap Q\ne \varnothing $, then
\begin{equation}\label{eq:property of Whitney cubes}
\begin{gathered}
\diam(Q)\leq r,\\
Q\subset B(y,3r),
\end{gathered}
\end{equation}
see \cite[Section 2.3]{tolsa2015characterization} for details.
We set 
\begin{equation*}
\W_{k}^e = \{ Q\in\W^e : \ell(Q)\leq 2^{-k}\},
\end{equation*}
and also, for every $Q\in\DG$ satisfying $\ell(Q)\le 2^{-k_0}$,
\begin{equation*}
\W_Q = \W_{k(Q)}^{e(Q)}.
\end{equation*}
\begin{remark}
	It follows immediately from the definition of $k(Q)$ that if $P\in\W_Q$, then
	\begin{equation*}
	\ell(P)\le 2^{-k(Q)} = 2^{k_0}\ell(Q).
	\end{equation*}
\end{remark}

\subsection{Constants and Parameters}\label{subsec:constants}
For reader's convenience, we collect here all the constants that appear in the proof. We indicate what depends on what, and when each constant gets fixed.

Recall that measure $\mu$ and Lipschitz graph $\Gamma$ were fixed at the very beginning, in Subsection \ref{subsec:notation}, and also that $\lip(\Gamma)\le 1$. Moreover, in \remref{remark:mu compactly supported} we fixed $R\in\D_{\Gamma}$ with $\ell(R)=1$, and without loss of generality we assumed that $\mu$ is finite and compactly supported.
\begin{itemize}
	\item $0<\varepsilon\ll 1$ is a constant from the assumptions of \lemref{lem:general_with_Gamma} and it was fixed in Subsection \ref{subsec:notation},
	\item $\Lambda$ is an absolute constant from the definition of $\widetilde{B}_Q=\Lambda B_Q$, it is fixed in \eqref{eq:fixing Lambda} (actually, one can take $\Lambda=9\sqrt{2}$),
	\item $k_0=k_0(n,\Lambda)$ is an integer from \lemref{lem:Qtilde_corresp_to_Q},
	\item $\varepsilon_0=\varepsilon_0(n)$ is the constant from \lemref{lem:orthogonal planes L},
	\item $K$ and $D_0$ are dimensional constants from the definition of Whitney cubes,
	\item $\lambda=\lambda(k_0,K,n,d)>3$ is fixed in \lemref{lem:estimate alpha mu with alpha sigma}, more precisely in equation \eqref{eq:support of muQ} (one can choose e.g. $\lambda=C(n,d)\,K\,2^{k_0}$),
	\item $M=M(\varepsilon,\lambda,\Lambda,n,d,\mu)\gg 1$ is chosen in \lemref{lem:fixing M}.
\end{itemize}

%	A subfamily of dyadic cubes $T\subset D_{\Gamma}(R)$ is called a tree with root $R$ if for every $Q\ne R$ belonging to $T$, the parent of $Q$ also belongs to $T$.

%	\section{Proof of Theorem \ref{thm:general_case} on Lipschitz graphs}
%	In this section we prove Theorem \ref{thm:general_case} in the special case of measures supported on Lipschitz graphs.
%	\begin{theorem}
%		Let $\Gamma\subset\R^d$ be an $n$-dimensional Lipschitz graph, $\mu=g\sigma$ for some $g\in L^1(\restr{\H^n}{\Gamma})$. Then for $\mu$-a.e. $x\in\Gamma$ we have
%		\begin{equation*}
%		\int_0^{\infty}\alpha_{\mu,2}(x,r)^2 \frac{dr}{r}.
%		\end{equation*}
%	\end{theorem}
%	
%	\begin{lemma}
%		Let $\Gamma\subset\R^d$ be an $n$-dimensional Lipschitz graph, $\mu=g\sigma$ for some $g\in L^1(\restr{\H^n}{\Gamma})$, $R\in D_{\Gamma}$, and $\varepsilon>0$. Then there exist a constant $M>0$ and a tree $T\subset D_{\Gamma}(R)$ with root $R$ such that the stopping region is small
%		\begin{equation}
%		\mu(\bigcup \mathrm{Stop}(T))<\varepsilon,
%		\end{equation}
%		we control the measure of $Q\in T$ 
%		\begin{gather}
%		\mu(B_Q)\le M\ell(Q)^n,\\
%		\mu(Q)\ge \frac{1}{M} \ell(Q)^n,
%		\end{gather}
%		and $\alpha_{\mu,2}(Q)$ satisfy the packing condition
%		\begin{equation}
%		\sum_{Q\in T}\amu(Q)^2\ell(Q)^n\lesssim \ell(R)^n.
%		\end{equation}
%	\end{lemma}
%\subsection*{Approximating measures}

\section{Estimates of \texorpdfstring{$\alpha_2$}{alpha\_2} Coefficients}\label{sec:estimates of aplha2}
%We say that an $n$-plane $L$ is $\varepsilon$-orthogonal to another $n$-plane $L'$ if there exists a distinct pair of points $x,y\in L$ such that for every distinct $w,z\in L'$ we have $|\langle \frac{x-y}{\lvert x-y\rvert}, \frac{w-z}{\lvert w-z\rvert}\rangle| < \varepsilon.$ If planes $L, L'$ are not $\varepsilon$-orthogonal, we will say that they are $\varepsilon$-far from being orthogonal.

We begin by showing the relationship between $b\beta_2$ and $\alpha_2$ coefficients.
\begin{lemma}\label{lem:alpha2 controls alpha and beta2}
	Suppose that $\nu$ is a Radon measure, $B$ is a ball satisfying $\nu(B)\approx r(B)^n$, and $L$ is a plane minimizing $\alpha_{\nu,2}(B)$. Then
	\begin{equation*}
	b\beta_{\nu,2}(B)^2\lesssim r(B)^{-n-2} \int_B \dist(x,L)^2\ d\nu \lesssim \alpha_{\nu,2}(B).
	\end{equation*}
\end{lemma}

\begin{proof}
	%The inequality $\beta_{\nu,2}(B)\le b\beta_{\nu,2}(B)$ is trivial. 
	Let $\pi$ be a minimizing transport plan between $\varphi_{{B}}\nu$ and $a_{B,L}\varphi_{{B}}\Hn{L}$ (where $a_{B,L}$ is as in the definition of $\alpha_{\nu,2}(B)$; note that $a_{B,L}\gtrsim 1$ since $\nu(B)\approx r(B)^n$). Then, by the definition of a transport plan, and the fact that $\varphi_B\equiv 1$ on $B$, 
	\begin{multline*}
	\alpha_{\nu,2}(B)^2r(B)^2\nu(B) = \int |x-y|^2\ d\pi(x,y)\\\ge \frac{1}{2}\int_B \dist(x,L)^2\ d\nu + \frac{a_{B,L}}{2}\int_B \dist(y,\supp\nu)^2\ d\Hn{L}
	\gtrsim b\beta_{\nu,2}(B)^2r(B)^{n+2}.
	\end{multline*}
\end{proof}

Recall that $\Gamma$ is an $n$-dimensional $1$-Lipschitz graph that was fixed in Subsection \ref{subsec:notation}, $\sigma=(\Pi_{\Gamma})_*\Hn{L_0}$, and that $L_Q$ is the plane minimizing $\alpha_{\sigma,2}(B_Q)$. The next lemma states that $\Gamma$-cubes $Q$ whose best approximating planes $L_Q$ form big angle with $L_0$ have large $\alpha_2$ numbers. In consequence, there are very few cubes of this kind (in fact, they form a Carleson family).
\begin{lemma}\label{lem:orthogonal planes L}
	There exists $\varepsilon_0=\varepsilon_0(n)>0$ such that for every $Q\in\EDG$ with $\measuredangle(L_Q,L_0)>1-\varepsilon_0$ we have
	\begin{equation*}
	\alpha_{\sigma,2}(B_Q) \gtrsim 1.
	\end{equation*}
\end{lemma}
\begin{proof}
	Suppose $Q\in\EDG$. Take $x_k\in 0.5B_Q\cap\Gamma,\ k=1,\dots,n,$ such that $|x_k-z_Q|=0.5r(B_Q),$ and the vectors $\{\Pi_0(x_k-z_Q)\}_k$ form an orthogonal basis of $L_0$. Set $B_0=B(z_Q,\eta r(B_Q)),\ B_k=B(x_k,\eta r(B_Q))$, where $\eta=\eta(n)<0.01$ is a small dimensional constant that will be chosen later. Clearly, for all $k=0,\dots,n$  we have $B_k\subset B_Q$.		
	
	If $L_Q$ does not intersect one of the balls, say $B_k$, then by \lemref{lem:alpha2 controls alpha and beta2}
	\begin{equation*}
	\alpha_{\sigma,2}(B_Q)^2 r(B_Q)^{n+2} \gtrsim \int_{B_Q} \dist(x,L_Q)^2\ d\sigma
	\ge\int_{\frac{1}{2}B_k}\dist(x,L_Q)^2\ d\sigma \gtrsim \eta^{n+2}r(B_Q)^{n+2}.
	\end{equation*} 
	
	Now suppose that $L_Q$ intersects all $B_k$. Then, since $B_k$ are all centered at $\Gamma$, $\Gamma$ is $1$-Lipschitz, and $x_k$ were chosen appropriately, it is easy to see that for $\eta=\eta(n)$ and $\varepsilon_0=\varepsilon_0(n)$ small enough we have $\measuredangle(L_Q,L_0)\le 1-\varepsilon_0$.
\end{proof}

The following two lemmas will let us compare $\alpha_2$ coefficients at similar scales, so that we can pass from the integral form of $\alpha_2$ square function \eqref{eq:alpha2 square function} to its dyadic variant.

\begin{lemma}[{\cite[Lemma 5.3]{tolsa2012mass}}]\label{lem:W_2 est Tolsa}
	Let $\nu$ be a finite measure supported inside the ball $B'\subset\R^d$. Let $B\subset\R^d$ be another ball such that $3B\subset B'$, with $r(B)\approx r(B')$ and $\nu(B)\approx \nu(B')\approx r(B)^n$. Let $L$ be an $n$-plane which intersects $B$ and let $f:L\rightarrow[0,1]$ be a function such that $f\equiv 1$ on $3B$, $f\equiv 0$ on $L\setminus B'$. Then
	\begin{equation*}
	W_2(\varphi_B\nu, a\varphi_B\Hn{L})\lesssim W_2(\nu, af\Hn{L}).
	\end{equation*}
\end{lemma}

Recall that $\widehat{\alpha}_{\mu,2}(B) = \alpha_{\mu,2,L_B}(B)$.

\begin{lemma}\label{lem:alpha_on_balls}
	Let $\nu$ be a Radon measure on $\R^d$, $B_1,B_2\subset\R^d$ be balls centered at $\Gamma$ with $3B_1\subset B_2,\ r(B_1)\approx r(B_2),\ \nu(B_1)\approx \nu(3B_2)\approx r(B_2)^n$. 
	%		Then for any $n$-plane $L$ intersecting $3B_1$ we have
	%		\begin{equation}
	%		W_2(\varphi_{B_1}\nu,a\varphi_{B_1}\Hn{L})\lesssim W_2(\varphi_2\nu,a\varphi_2\Hn{L}).
	%		\end{equation}
	Then we have
	\begin{gather}
	\widehat{\alpha}_{\nu,2}(B_1)\lesssim \widehat{\alpha}_{\nu,2}(B_2) + \alpha_{\sigma,2}(B_2).
	\end{gather}
\end{lemma}

\begin{proof}
	We begin by noting that since $\nu(3B_1)\lesssim \nu(B_1)$, we have $\widehat{\alpha}_{\nu,2}(B_1)\lesssim 1$. As a result, it suffices to prove the lemma under the assumption ${\alpha}_{\sigma,2}(B_2)\leq\delta$ for some small constant $\delta>0$ which will be fixed later on. 
	
	For brevity of notation set $\varphi_i = \varphi_{B_i},\ L_i = L_{B_i}$ for  $i=1,2.$ We want to apply \lemref{lem:W_2 est Tolsa} with $ B = B_1,\, B' = 3B_2,\, \nu = 
	%\frac{\int \varphi_2\ d\Hn{L_2}}{\int \varphi_2\ d\nu}
	\varphi_2\nu,\, L=L_2,\, f = \restr{\varphi_2}{L}.$ What needs to be checked is that $B_1\cap L_2\ne \varnothing $. If this intersection were empty, we would have by \lemref{lem:alpha2 controls alpha and beta2}
	\begin{multline*}
	{\alpha}_{\sigma,2}(B_2)^2r(B_2)^{n+2} \gtrsim \int_{B_2} \dist(x,L_2)^2\ d\sigma \ge \int_{B_1} \dist(x,L_2)^2\ d\sigma\\
	\ge \int_{\frac{1}{2}B_1} \frac{1}{2}r(B_1)^2\ d\sigma \approx r(B_1)^{n+2}\approx r(B_2)^{n+2}.
	\end{multline*}
	Thus, if $B_1\cap L_2= \varnothing $, then ${\alpha}_{\sigma,2}(B_2)\gtrsim 1$ and we arrive at a contradiction with ${\alpha}_{\sigma,2}(B_2)\leq\delta$ for $\delta$ small enough.
	
	So the assumptions of \lemref{lem:W_2 est Tolsa} are met and we get
	\begin{equation}\label{eq:using Lemma W_2 Tolsa}
	W_2(\varphi_1\nu,a\varphi_1\Hn{L_2})\lesssim W_2(\varphi_2\nu,a\varphi_2\Hn{L_2}).
	\end{equation}
	
	Similarly, taking $\nu = 
	%\frac{\int \varphi_2\ d\Hn{L}}{\int \varphi_2\ d\sigma}
	\varphi_2\sigma$ and $B = B_1,\, B' = 3B_2,\, L=L_2,\, f=\restr{\varphi_2}{L}$ it follows that 
	\begin{equation}\label{eq:using Lemma W_2 Tolsa again}
	W_2(\varphi_1\sigma,a\varphi_1\Hn{L_2})\lesssim W_2(\varphi_2\sigma,a\varphi_2\Hn{L_2}).
	\end{equation}
	
	Using the triangle inequality, the scaling of $W_2$, the fact that $L_1$ minimizes $\alpha_{\sigma,2}(B_1)$, and the inequalities above, we arrive at
	\begin{multline}\label{eq:ineq W_2 triangle ineq}
	W_2(\varphi_1\nu,a\varphi_1\Hn{L_1})\leq W_2(\varphi_1\nu,a\varphi_1\Hn{L_2}) \\
	+ \left(\frac{\int \varphi_1\ d\nu}{\int \varphi_1\ d\sigma}\right)^{1/2}\left(W_2(\varphi_1\sigma,a\varphi_1\Hn{L_1}) +
	W_2(\varphi_1\sigma,a\varphi_1\Hn{L_2})\right)\\
	\overset{L_1\, \text{minimizer}}{\lesssim}W_2(\varphi_1\nu,a\varphi_1\Hn{L_2}) + \left(\frac{\nu(3B_1)}{r(B_1)^n}\right)^{1/2}W_2(\varphi_1\sigma,a\varphi_1\Hn{L_2})\\
	\lesssim W_2(\varphi_1\nu,a\varphi_1\Hn{L_2}) +W_2(\varphi_1\sigma,a\varphi_1\Hn{L_2})\\
	\overset{\eqref{eq:using Lemma W_2 Tolsa},\eqref{eq:using Lemma W_2 Tolsa again}}{\lesssim} W_2(\varphi_2\nu,a\varphi_2\Hn{L_2}) +W_2(\varphi_2\sigma,a\varphi_2\Hn{L_2}).
	\end{multline}
	Dividing both sides by $r(B_1)^{1+n/2}$ yields
	\begin{equation*}
	\widehat{\alpha}_{\nu,2}(B_1)\lesssim \widehat{\alpha}_{\nu,2}(B_2) + \alpha_{\sigma,2}(B_2).
	\end{equation*}
\end{proof}

For technical reasons we define a modified version of $\alpha_2$ coefficients. For any $Q\in\EDG$ set
\begin{equation*}
\atilde_{\nu,2}(Q) = \begin{cases}
1\quad &\text{if}\ \measuredangle(L_Q,L_0)>1-\varepsilon_0,\\
\ell(Q)^{-(1+\frac{n}{2})}W_2(\psi_Q\nu,a\psi_Q\Hn{L_{Q}}) &\text{otherwise,}
\end{cases}
\end{equation*}
where $\varepsilon_0$ is as in \lemref{lem:orthogonal planes L}, and
\begin{gather*}
\psi_Q = \one_{V(Q)},\\		
a = \frac{\int \psi_Q\ d\nu}{\int \psi_Q\ d\Hn{L_Q}}.
\end{gather*}
Recall that $\sigma=(\Pi_{\Gamma})_*\Hn{L_0}\approx\Hn{\Gamma}$.
\begin{lemma}\label{lem:alphas on cubes}
	Let $\nu\ll\sigma$, $B\subset\R^d$ be a ball, $Q\in\EDG$. Suppose they satisfy $3B\subset V(Q)\cap B_Q,\ r(B)\approx \ell(Q),\ \nu(B)\approx\nu(Q)\approx \ell(Q)^n$. Then
	\begin{equation*}
	\widehat{\alpha}_{\nu,2}(B)\lesssim_{\varepsilon_0}\widetilde{\alpha}_{\nu,2}(Q) + \alpha_{\sigma,2}(B_Q).
	\end{equation*}
\end{lemma}

\begin{proof}
	Since $\nu(B)>0$ and $\supp\nu\subset\Gamma$, we certainly have $\sigma(3B)\approx r(B)^n$. Moreover, our assumptions imply that $\nu(3B)\approx\nu(B)$, and so $\widehat{\alpha}_{\nu,2}(B)\lesssim 1$. Thus, we may argue in the same way as in the beginning of the proof of \lemref{lem:alpha_on_balls} to conclude that, without loss of generality, $L_Q\cap B \ne \varnothing $. Similarly, we may assume that $\measuredangle(L_Q,L_0)\le 1-\varepsilon_0$, because otherwise it would follow from \lemref{lem:orthogonal planes L} that $\alpha_{\sigma,2}(B_Q)$ is big.
	
	Now, since $\measuredangle(L_Q,L_0)\le 1-\varepsilon_0$, we get that $V(Q)\cap L_{Q}\subset \kappa B_Q$ for some constant $\kappa$ depending on $\varepsilon_0$; we may assume $\kappa>10$.
	
	We use \lemref{lem:W_2 est Tolsa} twice, first with $B=B,\ B'=\kappa B_Q,\ \nu=
	%\frac{\int\psi_Q\ d\Hn{L_2}}{\int\psi_Q\ d\nu}
	\psi_Q\nu,\ L=L_{Q},\ f=\restr{\psi_Q}{L}$, and then with $B=B,\ B'=\kappa B_Q,\ \nu = \varphi_Q\sigma,\ L=L_{Q},\ f = \restr{\varphi_Q}{L}$, to obtain
	\begin{gather*}
	W_2(\varphi_B\nu,a\varphi_B\Hn{L_Q})\lesssim_{\kappa} W_2(\psi_Q\nu,a\psi_Q\Hn{L_Q}),\\
	W_2(\varphi_B\sigma,a\varphi_B\Hn{L_Q})\lesssim_{\kappa} W_2(\varphi_Q\sigma,a\varphi_Q\Hn{L_Q}).
	\end{gather*}
	By the triangle inequality, the scaling of $W_2$, the fact that $L_B$ minimizes $\alpha_{\sigma,2}(B)$, and the estimates above we get
	\begin{multline*}
	W_2(\varphi_B\nu,a\varphi_B\Hn{L_B})\le W_2(\varphi_B\nu,a\varphi_B\Hn{L_Q})\\ + \left(\frac{\int \varphi_B\ d\nu}{\int \varphi_B\ d\sigma}\right)^{1/2}\left(W_2(\varphi_B\sigma,a\varphi_B\Hn{L_B}) +
	W_2(\varphi_B\sigma,a\varphi_B\Hn{L_Q})\right)\\
	\lesssim W_2(\varphi_B\nu,a\varphi_B\Hn{L_Q}) + \left(\frac{\nu(3B)}{r(B)^n}\right)^{1/2}W_2(\varphi_B\sigma,a\varphi_B\Hn{L_Q})\\
	\lesssim_{\kappa} W_2(\psi_Q\nu,a\psi_Q\Hn{L_Q}) + W_2(\varphi_Q\sigma,a\varphi_Q\Hn{L_Q}).
	\end{multline*}
	Dividing both sides by $r(B)^{1+n/2}$ yields the desired result.
\end{proof}

We will need an estimate which is a slight modification of {\cite[Lemma 6.2]{tolsa2012mass}}. In order to formulate it, let us introduce the usual martingale difference operator. Recall that if $P\in \DG^e$ for some $e\in\{0,1\}^n$, then $P'\in \DG^e$ is a child of $P$ if $P'\subset P$ and $\ell(P')=\frac{1}{2}\ell(P)$. Children of $P\in \D_{\R^n}^e$ are defined analogously.

Given $g\in L^1_{loc}(\sigma)$ and $P\in \DG^e$ we set 
\begin{equation*}
\Delta^{\sigma}_P g(x) = 	\begin{cases}
\frac{\int_{P'} g\ d\sigma}{\sigma(P')} - \frac{\int_{P} g\ d\sigma}{\sigma(P)}\  &: x\in P',\ P'\ \text{a child of}\ P,\\
0\quad &: x\not\in P.
\end{cases}
\end{equation*}
Given $h\in L^1_{loc}(\Hn{L_0})$ and $P\in \D^e_{\R^n}$ we define analogously $\Delta_P h(x)$:
\begin{equation*}
\Delta_P h(x) = 	\begin{cases}
\frac{\int_{P'} h\ d\H^n}{\ell(P')^n} - \frac{\int_{P} h\ d\H^n}{\ell(P)^n}\  &: x\in P',\ P'\ \text{a child of}\ P,\\
0\quad &: x\not\in P.
\end{cases}
\end{equation*}
Recall that for $g\in L^2(\sigma)$ we have
\begin{equation*}
g = \sum_{P\in\DGe} \Delta^{\sigma}_P g,
\end{equation*}
in the sense of $L^2(\sigma)$, and 
\begin{equation*} 
\lVert g\rVert_{L^2(\sigma)}^2 = \sum_{P\in\DGe}\lVert \Delta^{\sigma}_P g\rVert^2_{L^2(\sigma)},
\end{equation*}
for details see e.g. \cite[Part I]{david1991wavelets} or \cite[Section 5.4.2]{grafakos2008classical}.

Let us introduce also some additional vocabulary. We will say that a family of cubes $\Tree\subset\DG^e$ is a tree with root $R_0$ if it satisfies:
\begin{enumerate}
	\item[(T1)] $R_0\in\Tree$, and for every $Q\in\Tree$ we have $Q\subset R_0$,
	\item[(T2)] for every $Q\in\Tree$ such that $Q\neq R_0$, the parent of $Q$ also belongs to $\Tree$.
\end{enumerate}
By iterating (T2), we can actually see that if $Q\in\Tree$, then all the intermediate cubes $Q\subset P\subset R_0$ also belong to $\Tree$. 

The stopping region of $\Tree$, denoted by $\Stop(\Tree)$, is the family of {all} the cubes $P\in \DGe(R_0)$ satisfying:
\begin{enumerate}
	\item[(S)] $P\not\in\Tree$, but the parent of $P$ belongs to $\Tree$.
\end{enumerate}
It is easy to see that the cubes from $\Stop(\Tree)$ are pairwise disjoint, and that they are maximal descendants of $R_0$ not belonging to $\Tree$. Moreover, for every $x\in R_0$ we have either $x\in P$ for some $P\in\Stop(\Tree)$, or $x\in Q_k$ for a sequence  of cubes $\{Q_k\}_k\subset\Tree$ satisfying $\ell(Q_k)\xrightarrow{k\to\infty}0.$

The following lemma is a modified version of {\cite[Lemma 6.2]{tolsa2012mass}}.
\begin{lemma} \label{lem:estimate alpha tilde from GAFA}
	Let $\nu$ be a Radon measure on $\Gamma$ of the form $\nu = g\sigma$, with $g\in L^1(\sigma),\ 0\le g\le C$ for some $C>1$. Consider a cube $Q\in\EDG$ and a tree $\T$ with root $Q$. Suppose that for all $P\in\T$ we have $C^{-1}\ell(P)^n\le \nu(P)\le C\ell(P)^n$. Then, we have
	\begin{equation}\label{eq:estimate alpha tilde 2}
	\widetilde{\alpha}_{\nu,2}(Q)^2\lesssim_{\varepsilon_0,C} \alpha_{\sigma,2}(B_Q)^2 + \sum_{P\in\T}\lVert \Delta^{\sigma}_P g\rVert ^2_{L^2(\sigma)}\frac{\ell(P)}{\ell(Q)^{n+1}} + \sum_{S\in \Stop(\T)}\frac{\ell(S)^2}{\ell(Q)^{n+2}}\nu(S),
	\end{equation}
	and
	\begin{equation}\label{eq:sum of martingales on tree}
	\sum_{P\in\T}\lVert \Delta^{\sigma}_P g\rVert ^2_{L^2(\sigma)}\le C\lVert g\rVert_{L^1(\sigma)}=C \nu(\Gamma).
	\end{equation}
\end{lemma}

In the proof we will use \cite[Remark 3.14]{tolsa2012mass}. It can be thought of as a flat counterpart of \lemref{lem:estimate alpha tilde from GAFA} -- it is valid for more general measures $\nu$ (even more general then what we state below), but at the price of assuming $\Gamma=L_0\simeq\R^n$.
\begin{lemma}[simplified {\cite[Remark 3.14]{tolsa2012mass}}]\label{lem:remark 3.14 from Tolsa12}
	Suppose $Q\in\D_{\R^n}$ is a dyadic cube in $\R^n$ and $\T$ is a tree with root $Q$. Consider a measure $\nu=g\Hn{Q}$ such that $\nu(P)\approx \ell(P)^n$ for $P\in\Tree$. Then,
	\begin{equation*}
	W_2(\nu,a\Hn{Q})\lesssim \sum_{P\in\T}\lVert \Delta_P g\rVert ^2_{L^2(\H^n)}\ell(P)\ell(Q) + \sum_{S\in \Stop(\T)}\ell(S)^2\nu(S).
	\end{equation*}
\end{lemma}

\begin{remark}\label{rem:different trees}
	The definition of a tree of dyadic cubes in \cite[p. 492]{tolsa2012mass} is slightly more restrictive than the one we adopted. Apart from conditions (T1) and (T2), they also satisfy
	\begin{itemize}
		\item[(T3)] if $Q\in\Tree$, then either all the children of $Q$ belong to $\Tree$, or none of them.
	\end{itemize}
	Equivalently, if $Q\in\Tree$, and $Q$ is not the root, then all the brothers of $Q$ also belong to $\Tree$. To underline the difference between the two notions, sometimes the terms \emph{coherent} and \emph{semicoherent} family of cubes are used. The former refers to trees satisfying (T1--T3), the latter to those satisfying (T1--T2).
	
	Nevertheless, \cite[Remark 3.14]{tolsa2012mass} cited above is true for both coherent and semicoherent families of cubes. That is, property (T3) is never used in the proof of either \cite[Remark 3.14]{tolsa2012mass} or the preceding ``key lemma'' \cite[Lemma 3.13]{tolsa2012mass}.
\end{remark}

We are finally ready to prove \lemref{lem:estimate alpha tilde from GAFA}.

\begin{proof}[Proof of \lemref{lem:estimate alpha tilde from GAFA}]
	Let $L=L_Q$. If $\measuredangle(L,L_0)>1-\varepsilon_0$, then by \lemref{lem:orthogonal planes L} and the definition of $\widetilde{\alpha}_{\nu,2}(Q)$
	\begin{equation*}
	\widetilde{\alpha}_{\nu,2}(Q)^2 = 1 \lesssim \alpha_{\sigma,2}(B_Q)^2,
	\end{equation*}
	and we are done. Now assume that $\measuredangle(L,L_0)\le 1-\varepsilon_0$.
	
%	\emph{Step 1.} Assume first that the density $g$ is not only in $L^1(\sigma)$, but that it also satisfies
%	\begin{equation}\label{eq:bounded density}
%	\lVert g\rVert_{L^{\infty}(\sigma)}\le C.
%	\end{equation}
%	
	Let $\tilde{\Pi}_L$ be the projection from $\R^d$ onto $L$, orthogonal to $L_0$. We also consider the flat measure $\sigma_L = (\tilde{\Pi}_L)_*\sigma = (\tilde{\Pi}_L)_*\Hn{L_0} = c_L\Hn{L}$ (recall that $\Pi_{\Gamma}$ is a projection orthogonal to $L_0$, so that $\tilde{\Pi}_L\circ\Pi_{\Gamma} = \tilde{\Pi}_L$). Define $g_0:L_0\rightarrow\R$ as $g_0=g\circ\Pi_{\Gamma}$. 
	
	By triangle inequality
	\begin{equation}\label{eq:estimate W_2 nu with projections}
	W_2(\psi_Q\nu,a\psi_Q\Hn{L}) = W_2(\psi_Q\nu,a\psi_Q\sigma_L)\le W_2(\psi_Q\nu,\psi_Q(\tilde{\Pi}_L)_*\nu) + W_2(\psi_Q(\tilde{\Pi}_L)_*\nu,a\psi_Q\sigma_L).
	\end{equation}
	The first term from the right hand side is estimated by $\alpha_{\sigma,2}(B_Q)$:
	\begin{multline*}
	W_2(\psi_Q\nu,\psi_Q(\tilde{\Pi}_L)_*\nu)^2\le \int_Q|x-\tilde{\Pi}_L(x)|^2\ d\nu(x)\approx_{\varepsilon_0} \int_Q \dist(x,L)^2\ d\nu(x)\\
	{\lesssim}_C \int_Q \dist(x,L)^2\ d\sigma(x)\lesssim \alpha_{\sigma,2}(B_Q)^2\ell(Q)^{n+2}.
	\end{multline*}
	We estimate the second term from the right hand side of \eqref{eq:estimate W_2 nu with projections} using the fact that $\restr{\Pi_0}{L\cap V(Q)} : L\cap V(Q) \rightarrow L_0\cap V(Q)$ is bilipschitz, with a constant depending on $\varepsilon_0$ (because $\measuredangle(L,L_0)\le 1-\varepsilon_0$):
	\begin{multline*}
	W_2(\psi_Q(\tilde{\Pi}_L)_*\nu,a\psi_Q\sigma_L)\approx_{\varepsilon_0} W_2(\psi_Q(\Pi_0)_*((\tilde{\Pi}_L)_*\nu),a\psi_Q(\Pi_0)_*\sigma_L)\\
	= W_2(\psi_Q g_0\Hn{L_0},a\psi_Q\Hn{L_0}).
	\end{multline*}
	By \lemref{lem:remark 3.14 from Tolsa12} we have
	\begin{equation*}
	W_2(\psi_Q g_0\Hn{L_0},a\psi_Q\Hn{L_0})^2 \lesssim \sum_{P'\in\T_{\R^n}}\lVert\Delta_{P'} g_0\rVert^2_{L^2(L_0)}\ell(P')\ell(Q) + \sum_{S\in\Stop(\T)}\ell(S)^2\nu(S),
	\end{equation*}
	where $\T_{\R^n}\subset\D_{\R^n}$ is the tree formed by cubes $P'=\Pi_0(P),\ P\in\T$, and $L^2(L_0)=L^2(\Hn{L_0})$.
	
	Using \eqref{eq:estimate W_2 nu with projections} and the estimates above we get
	\begin{multline*}
	W_2(\psi_Q\nu,a\psi_Q\Hn{L})^2\\
	\lesssim_{\varepsilon_0} \alpha_{\sigma,2}(B_Q)^2\ell(Q)^{n+2} + \sum_{P'\in\T_{\R^n}}\lVert\Delta_{P'} g_0\rVert^2_{L^2(L_0)}\ell(P')\ell(Q) + \sum_{S\in\Stop(\T)}\ell(S)^2\nu(S).
	\end{multline*}
	We conclude the proof of \eqref{eq:estimate alpha tilde 2} by noting that for each $P\in\T$
	\begin{equation*}
	\lVert \Delta^{\sigma}_P g\rVert_{L^2(\sigma)} = \lVert \Delta_{\Pi_0(P)} g_0\rVert_{L^2(L_0)}.
	\end{equation*}
	The estimate \eqref{eq:sum of martingales on tree} follows trivially from the fact that if $e\in\{0,1\}^n$ is such that $Q\in\DG^e$, then
	\begin{equation*}
	\sum_{P\in\T}\lVert \Delta^{\sigma}_P g\rVert ^2_{L^2(\sigma)}\le \sum_{P\in\DG^e}\lVert \Delta^{\sigma}_P g\rVert ^2_{L^2(\sigma)} = \lVert g\rVert^2_{L^2(\sigma)}\le C\lVert g\rVert_{L^1(\sigma)}.
	\end{equation*}
\end{proof}

We would like to use \lemref{lem:estimate alpha tilde from GAFA} also on measures with unbounded density. An approximation argument allows us to get rid of the boundedness assumption, at least if we assume additionally that $\nu(B_P)\le C\ell(P)^n$ for $P\in\Tree$.

\begin{lemma}\label{lem:estimate alpha tilde unbdd density}
	Let $\nu = g\sigma$ with $g\in L^1(\sigma),\ g\ge 0$. Consider a cube $Q\in\EDG$ and a tree $\T$ with root $Q$. Suppose there exists $C>1$ such that for all $P\in\T$ we have $C^{-1}\ell(P)^n\le \nu(P)\le\nu(B_P)\le C\ell(P)^n$. Then, we have
	\begin{equation}\label{eq:estimate alpha tilde unbdd density}
	\widetilde{\alpha}_{\nu,2}(Q)^2\lesssim_{\varepsilon_0,C} \alpha_{\sigma,2}(B_Q)^2 + \sum_{P\in\T}\lVert \Delta^{\sigma}_P g\rVert ^2_{L^2(\sigma)}\frac{\ell(P)}{\ell(Q)^{n+1}} + \sum_{S\in \Stop(\T)}\frac{\ell(S)^2}{\ell(Q)^{n+2}}\nu(S),
	\end{equation}
	and
	\begin{equation}\label{eq:sum of martingales on tree unbdd density}
	\sum_{P\in\T}\lVert \Delta^{\sigma}_P g\rVert ^2_{L^2(\sigma)}\le C\lVert g\rVert_{L^1(\sigma)}=C \nu(\Gamma).
	\end{equation}
\end{lemma}

We divide the proof into smaller pieces. Let $\Stop=\Stop(\Tree)$. First, we define the set of good points as
\begin{equation*}
G = Q\setminus \bigcup_{P\in\Stop}P.
\end{equation*}
Note that the points from $x\in G$ are not contained in any stopping cube, and so there are arbitrarily small cubes $P\in\T$ containing $x$. We introduce the following approximating measure:
\begin{equation*}
\tilde{\nu} = \restr{\nu}{G} + \sum_{S\in \Stop } \frac{\nu(S)}{\sigma(S)}\restr{\sigma}{S}.
\end{equation*}
It is clear that for $Q\in\Tree\cup\Stop$ we have $\tilde{\nu}(Q)=\nu(Q)$. Moreover, for $Q\in\Tree$
\begin{equation}\label{eq:Tree density estimate}
C^{-1}\ell(Q)^n\le \tilde{\nu}(Q)=\nu(Q)\le C\ell(Q)^n.
\end{equation}
On the other hand, each $S\in\Stop$ is a child of some $Q\in\Tree$, so that
\begin{equation}\label{eq:Stop density estimate}
\tilde{\nu}(S)=\nu(S)\le \nu(Q)\le C\ell(Q)^n = 2^n C\ell(S)^n.
\end{equation}

\begin{lemma}\label{lem:density bounded}
	We have	
	\begin{equation*}
	\left\lVert\frac{d\tilde{\nu}}{d\sigma}\right\rVert_{L^{\infty}(\sigma)}\lesssim C.
	\end{equation*}
\end{lemma}
\begin{proof}
	It is trivial that for $x\in S\in\Stop$ the density is constant and
	\begin{equation*}
	\frac{d\tilde{\nu}}{d\sigma}(x)=\frac{\nu(S)}{\sigma(S)}=\frac{\nu(S)}{\ell(S)^n}\overset{\eqref{eq:Stop density estimate}}{\le} 2^n\, C.
	\end{equation*}
	On the other hand, by the definition of $\tilde{\nu}$, for $\sigma$-a.e. $x\in G$ we have $\frac{d\tilde{\nu}}{d\sigma}(x)=\frac{d\nu}{d\sigma}(x)=g(x)$. Moreover, for $\sigma$-a.e. $x\in G$ we have a sequence of cubes $Q_j\in\T$ such that $\ell(Q_j)=2^{-j}$ and $x\in Q_j$. Note that there exists some integer $j_0>0$ (depending on dimension) such that 
	\begin{equation*}
	Q_{j+j_0}\subset B(x,2^{-j})\subset B_{Q_j}.
	\end{equation*}
	It follows that
	\begin{equation*}
	\frac{d\tilde{\nu}}{d\sigma}(x)=\frac{d\nu}{d\sigma}(x)=\lim_{j\to\infty}\frac{\nu(B(x,2^{-j}))}{\sigma(B(x,2^{-j}))}
	\le \lim_{j\to\infty}\frac{\nu(B_{Q_j})}{\sigma(Q_{j+j_0})}\le \lim_{j\to\infty}\frac{C\ell(Q_j)^n}{\ell(Q_{j+j_0})^n}=C\,2^{nj_0}.
	\end{equation*}
	Thus,
	\begin{equation*}
	\left\lVert\frac{d\tilde{\nu}}{d\sigma}\right\rVert_{L^{\infty}(\sigma)}\lesssim C.
	\end{equation*}
\end{proof}

Let $\tilde{g}\in L^1(\sigma)\cap L^{\infty}(\sigma)$ be such that $\tilde{\nu}=\tilde{g}\sigma$. Applying \lemref{lem:estimate alpha tilde from GAFA} to $\tilde{\nu}$ yields 
\begin{equation}\label{eq:approximating estimate 1}
\widetilde{\alpha}_{\tilde{\nu},2}(Q)^2\lesssim_{\varepsilon_0,C} \alpha_{\sigma,2}(B_Q)^2 + \sum_{P\in\T}\lVert \Delta^{\sigma}_P \tilde{g}\rVert ^2_{L^2(\sigma)}\frac{\ell(P)}{\ell(Q)^{n+1}} + \sum_{S\in \Stop}\frac{\ell(S)^2}{\ell(Q)^{n+2}}\tilde{\nu}(S),
\end{equation}
and
\begin{equation}\label{eq:approximating estimate 2}
\sum_{P\in\T}\lVert \Delta^{\sigma}_P \tilde{g}\rVert ^2_{L^2(\sigma)}\le C\lVert \tilde{g}\rVert_{L^1(\sigma)}=C \tilde{\nu}(\Gamma)=C\nu(\Gamma).
\end{equation}

Observe that for $P\in\Tree$ we have
\begin{equation}\label{eq:martingales the same}
\Delta^{\sigma}_P \tilde{g} =  \Delta^{\sigma}_P {g}.
\end{equation}
Indeed, for $x\notin P$ both quantities are equal to zero. For $x\in P'\subset P$, where $P'$ is a child of $P$, we have $P'\in\Tree\cup\Stop$, and so
\begin{equation*}
\Delta^{\sigma}_P \tilde{g}(x) = \frac{\int_{P'} \tilde{g}\ d\sigma}{\sigma(P')} - \frac{\int_{P} \tilde{g}\ d\sigma}{\sigma(P)} = 
\frac{\tilde{\nu}(P')}{\sigma(P')} - \frac{\tilde{\nu}(P)}{\sigma(P)} = \frac{{\nu}(P')}{\sigma(P')} - \frac{{\nu}(P)}{\sigma(P)}= \Delta^{\sigma}_P {g}.
\end{equation*}
Hence, \eqref{eq:sum of martingales on tree unbdd density} follows immediately from \eqref{eq:approximating estimate 2}.

Since for $S\in\Stop$ we have $\tilde{\nu}(S)=\nu(S)$, we can use \eqref{eq:martingales the same} to transform \eqref{eq:approximating estimate 1} into 
\begin{equation}\label{eq:approximating estimate 3}
\widetilde{\alpha}_{\tilde{\nu},2}(Q)^2\lesssim_{\varepsilon_0,C} \alpha_{\sigma,2}(B_Q)^2 + \sum_{P\in\T}\lVert \Delta^{\sigma}_P {g}\rVert ^2_{L^2(\sigma)}\frac{\ell(P)}{\ell(Q)^{n+1}} + \sum_{S\in \Stop}\frac{\ell(S)^2}{\ell(Q)^{n+2}}{\nu}(S).
\end{equation}
In order to reach \eqref{eq:estimate alpha tilde unbdd density} and finish the proof of \lemref{lem:estimate alpha tilde unbdd density}, we only need to show how to pass from the estimate on $\widetilde{\alpha}_{\tilde{\nu},2}(Q)$ \eqref{eq:approximating estimate 3} to one on $\widetilde{\alpha}_{{\nu},2}(Q)$.

\begin{proof}[Proof of \lemref{lem:estimate alpha tilde unbdd density}]	
	Recall that if $\measuredangle(L_Q,L_0)>1-\varepsilon_0$, then $\widetilde{\alpha}_{\nu,2}(Q)=1$, but at the same time $\alpha_{\sigma,2}(B_Q)\gtrsim 1$ by \lemref{lem:orthogonal planes L}, so this case is trivial. Suppose $\measuredangle(L_Q,L_0)\le 1-\varepsilon_0$.
	We define a transport plan between $\psi_Q \tilde{\nu}$ and $\psi_Q\nu$:
	\begin{equation*}
	d\pi(x,y) = \one_{Q\cap G}(x)d\nu(x) d\delta_{x}(y) + \sum_{S\in\Stop}\frac{\one_S(x)\one_S(y)}{\sigma(S)}d\nu(x) d\sigma(y),
	\end{equation*}
	and we estimate
	\begin{equation*}
	W_2(\psi_Q\tilde{\nu},\psi_Q\nu)^2\le \int|x-y|^2\ d\pi(x,y)\lesssim \sum_{S\in\Stop}\ell(S)^2\nu(S).
	\end{equation*}
	From the triangle inequality, the bound above, and \eqref{eq:approximating estimate 3}, we get that
	\begin{multline*}
	\widetilde{\alpha}_{\nu,2}(Q)^2 \approx \ell(Q)^{-(n+2)}W_2(\psi_Q\nu,a\psi_Q\Hn{L_Q})^2\\
	\lesssim 	\ell(Q)^{-(n+2)}\big(W_2(\psi_Q\tilde{\nu},\psi_Q\nu)^2 + W_2(\psi_Q\tilde{\nu},a\psi_Q\Hn{L_Q})^2\big)\\
	\lesssim_{\varepsilon_0,C} \alpha_{\sigma,2}(B_Q)^2 + \sum_{P\in\T}\lVert\Delta^{\sigma}_P g\rVert^2_{L^2(\sigma)}\frac{\ell(P)}{\ell(Q)^{n+1}} + \sum_{S\in \Stop}\frac{\ell(S)^{2}}{\ell(Q)^{n+2}}\nu(S).
	\end{multline*}
\end{proof}

\section{Approximating Measures}\label{sec:approximating measures nuQ}
% \texorpdfstring{$\nu_Q$}{nu\_Q}
We will construct a family of measures on $\Gamma$ that will approximate $\mu$. For every Whitney cube $P\in\W^e$ we define $g_P:\Gamma\rightarrow\R$ as
\begin{equation*}
g_P(x) = \frac{\mu(P)}{\ell(P)^n}\one_{\Pi_{\Gamma}(P)}(x).
\end{equation*}
Note that $\int g_P\ d\sg = \mu(P)$. 

Given $e\in\{0,1\}^n,\ k\in\Z$, we define the following measures supported on $\Gamma$:
\begin{align*}
\nu^e &= \restr{\mu}{\Gamma} + \left(\sum_{P\in\W^e}g_P\right)\sg,\\
\nu^e_k &= \restr{\mu}{\Gamma} + \left(\sum_{P\in\W_k^e}g_P\right)\sg.
\end{align*}
Moreover, for every $Q\in\DG$ with $\ell(Q)\le 2^{-k_0}$ we set
\begin{equation*}
\nu_Q = \nu^{e(Q)}_{k(Q)} = \restr{\mu}{\Gamma} + \left(\sum_{P\in\W_Q}g_P\right)\sg.
\end{equation*}
Note that, since we assume $\mu$ is finite and compactly supported (see Remark \ref{remark:mu compactly supported}), all the measures $\nu^e,\ \nu^e_k,$ are also finite and compactly supported.

We defined $\nu_Q$ in such a way that, for ``good'' $Q\in\DG$, the measures $\restr{\mu}{B_Q}$ and $\restr{\nu_Q}{B_Q}$ are close in the $W_2$ distance. This will be shown in Section \ref{sec:pass from approximating to mu}. The rest of this section is dedicated to the construction of a tree of ``good cubes''.

Recall that $R\in\D_{\Gamma}$ is a $\Gamma$-cube fixed in \remref{remark:mu compactly supported}, and $0<\varepsilon\ll 1$ is a small constant fixed in Subsection \ref{subsec:notation}.

\begin{lemma}\label{lem:T_lemma_general}
	 Let $\lambda>3$. Then, there exist a big constant $M=M(\varepsilon,\lambda,\Lambda,n,d,\mu)\gg 1$ and a tree of good cubes $\T=\Tree(\lambda,\varepsilon,M)\subset \DG(R)$ with root $R$, such that for every $Q\in \T$ we have
	\begin{align*}
	%\sigma(3\widetilde{B}_Q)\le M\ell(Q)^n,\\
	\mu(\lambda\widetilde{B}_Q)&\le M\ell(Q)^n,\\
	\mu(Q)&\ge M^{-1} \ell(Q)^n,
	\end{align*}
	the stopping region $\Stop = \Stop(\Tree)$ is small:
	\begin{equation*}
	\mu\bigg(\bigcup_{Q\in\Stop} Q\bigg)<\varepsilon,
	\end{equation*}
	and $\widehat{\alpha}_{\nu_Q,2}(\widetilde{B}_Q)^2$ satisfy the packing condition:
	\begin{equation}\label{eq:packing condition on tree}
	\sum_{\substack{Q\in \T}} \widehat{\alpha}_{\nu_Q,2}(\widetilde{B}_Q)^2 \ell(Q)^n <\infty.
	\end{equation}
\end{lemma}

We split the proof into several small lemmas. First, we define auxiliary families of good cubes in $\DGe$ using a standard stopping time argument. 

%%The stopping region $\mathrm{Stop}(\T)$ will be used later on to define the small exceptional set $H\subset R$, while the properties of good cubes $Q\in\T$ will allow us to prove \lemref{lem:general_with_Gamma}.
%
%
%\begin{proof}[Proof of \lemref{lem:T_lemma_general}]
%	\textbf{Step 1.} We begin by defining trees of good cubes in $\DGe$ using a standard stopping time argument. 
%	
	For each $e\in\{0,1\}^n$ there exists a finite collection of cubes $\{R_i^e\}\subset\DGe$ such that $\ell(R_i^e)=1,\ R_i^e\cap R\ne \varnothing.$ Set $R^e= \bigcup_i R_i^e$. Let $M\gg 1$ be constant to be fixed later on, and set
	\begin{align*}
	\HD_{\nu,0}^{e} &= \{Q\in\DGe :\  Q\subset R^e,\ \nu^e(\lambda\widetilde{B}_Q)>M\ell(Q)^n \},\\
	\HD_{\mu,0}^{e} &= \{Q\in\DGe :\  Q\subset R^e,\ \mu(\lambda\widetilde{B}_Q)>M\ell(Q)^n \},\\
	\LD_{0}^e &= \{Q\in\DGe :\  Q\subset R^e,\ \mu(Q)<M^{-1}\ell(Q)^n\}.
	\end{align*}
	$\HD$ and $\LD$ stand for ``high density'' and ``low density''. Let $\Stop^e\subset\DGe$ be the family of maximal with respect to inclusion cubes from $\HD_{\nu,0}^{e}\cup \HD_{\mu,0}^{e}\cup \LD_{0}^e$, and set $\HD_{\nu}^e= \HD_{\nu,0}^{e}\cap\Se,\ \HD_{\mu}^e= \HD_{\mu,0}^{e}\cap\Se,\ \LD^e= \LD_{0}^e\cap\Se$. Note that cubes from $\Se$ are pairwise disjoint. We define $\T^e$ as the family of those cubes from $\bigcup_i\DGe(R^e_i)$ which are not contained in any cube from $\Se$. Actually, this might not be a tree, but it is a finite collection of trees with roots $R_i^e$.
	
\begin{lemma}\label{lem:fixing M}
	For $M=M(\varepsilon,\lambda,\Lambda,n,d,\mu)$ big enough, we have for all $e\in\{0,1\}^n$
	\begin{equation}\label{eq:Stope small}
	\mu\bigg(\bigcup_{Q\in\Se}Q\bigg)<\frac{\varepsilon}{2^n}.
	\end{equation}
\end{lemma}
\begin{proof}
	Let $e\in\{0,1\}^n$. It is easy to see that the measure of $\LD^e$ is small: for every $Q\in \LD^{e}$ we have $\mu(Q)\le M^{-1}\sigma(Q)$, so 
	\begin{equation}\label{eq:LD est}
	\mu\bigg(\bigcup_{Q\in \LD^e}Q\bigg)\le  M^{-1}\sigma(R^e)\approx M^{-1}.
	\end{equation}		
	
	To estimate the measure of $\HD_{\mu}^e$, define for some big $N\gg 1$
	\begin{equation*}
	H_N = \{x\in\R^d : \mu(B(x,r))> Nr^n\ \text{for some}\ r\in(0,1)\}.
	\end{equation*}
	Since $\mu$ is $n$-rectifiable, the density $\Theta^n(x,\mu)$ exists, and is positive and finite $\mu$-a.e. Moreover, recall that $\mu(\R^d)$ is finite. This implies that for $N=N(\mu,\varepsilon,n)$ big enough
	\begin{equation*}
	\mu(H_N)\le\frac{\varepsilon}{2^{n+2}}.
	\end{equation*}
	
	We will show that, if $M$ is chosen big enough, then for all $Q\in \HD^e_{\mu}$ we have $Q\subset H_N $. Indeed, let $x\in Q\in \HD^e_{\mu}$. Then $B(x,2\lambda r(\widetilde{B}_Q)) \supset \lambda\widetilde{B}_Q$, and so
	\begin{equation*}
	\mu(B(x,2\lambda r(\widetilde{B}_Q)))\ge \mu(\lambda\widetilde{B}_Q)> M\ell(Q)^n > N (6\lambda\Lambda \diam(Q))^n=N (2\lambda r(\widetilde{B}_Q))^n,
	\end{equation*}
	for $M$ big enough with respect to $N, \lambda, \Lambda, n$. Moreover, note that for $Q\in \HD^e_{\mu}$ we have
	\begin{equation*}
	\frac{\mu(\R^d)}{M}>\ell(Q)^n\approx_{\Lambda} r(\widetilde{B}_Q)^n,
	\end{equation*}
	and so taking $M$ big enough (depending on $\mu(\R^d), \lambda,\Lambda, n$) we can ensure that all $Q\in \HD^e_{\mu}$ satisfy $2\lambda r(\widetilde{B}_Q)<1$. Thus, $x\in H_N$, and we conclude that 
	\begin{equation}\label{eq:HDmu est}
	\mu\bigg(\bigcup_{Q\in \HD^e_{\mu}} Q\bigg)\le\mu(H_N)\le\frac{\varepsilon}{2^{n+2}}.
	\end{equation}
	Since $\nu^e$ is a finite $n$-rectifiable measure, we can argue in the same way as above to get
	\begin{equation*}\label{eq:HDnu est}
	\nu^e\bigg(\bigcup_{Q\in \HD^e_{\nu}} Q\bigg)\le\frac{\varepsilon}{2^{n+2}}.
	\end{equation*}
	Smallness of $\mu(\bigcup_{Q\in \HD_{\nu}^e} Q)$ follows from the fact that $\restr{\mu}{\Gamma}\le\nu^e$.
	Putting this together with \eqref{eq:LD est} and \eqref{eq:HDmu est} we get
	\begin{equation*}
	\mu\bigg(\bigcup_{Q\in\Se}Q\bigg)<\frac{\varepsilon}{2^n}.
	\end{equation*}
	We take $M$ so big that the above holds for all $e\in\{0,1\}^n$, and the proof is finished.
\end{proof}
	
%	Set
%	\begin{equation*}
%	R^e_G = R^e\setminus \bigcup_{P\in\Se} P,
%	\end{equation*}
%	where $G$ stands for ``good''. Note that the points from $x\in R^e_G$ are not contained in any stopping cube, and so there are arbitrarily small cubes $Q\in\T^e$ containing $x$. 
	
	For each $e\in\{0,1\}^n,\ k= 0, 1, 2,\dots,$ let $g^e_k$ be the density of $\nu^e_k$ with respect to $\sigma$. Note that, due to the definition of $\Tree^e$, for any $Q\in\Tree^e$ we have
	\begin{equation*}
	M^{-1}\,\ell(Q)^n\le \nu^e_k(Q)\le \nu^e_k(B_Q)\le M\,\ell(Q)^n.
	\end{equation*} 
	Hence, given a cube $Q\in\T^e$ with $\ell(Q)=2^{-k}$, we can estimate $\widetilde{\alpha}_{\nu^e_k,2}(Q)^2$ using \lemref{lem:estimate alpha tilde unbdd density} (applied to $\nu^e_k$ and $\Tree= \{P\in \T^e\ :\ P\subset Q \}$) to get 
	\begin{equation}\label{eq:approximating estimate 4}
	\widetilde{\alpha}_{\nu^e_k,2}(Q)^2\lesssim_{\varepsilon_0,M} \alpha_{\sigma,2}(B_Q)^2 + \sum_{\substack{P\in\T^e\\P\subset Q}}\lVert \Delta^{\sigma}_P g^e_k\rVert ^2_{L^2(\sigma)}\frac{\ell(P)}{\ell(Q)^{n+1}} + \sum_{\substack{S\in\Se\\S\subset Q}}\frac{\ell(S)^2}{\ell(Q)^{n+2}}\nu^e_k(S).
	\end{equation}
	The following lemma states that the right hand side of this estimate can be made independent of $k$.
	
	\begin{lemma}\label{lem:another alpha tilde estimate}
	For all $Q\in\Tree^e$ with $\ell(Q)=2^{-k},\ k\ge 0,$ we have
	\begin{equation}\label{eq:final estimate of alpha tilde}
		\widetilde{\alpha}_{{\nu}^e_k,2}(Q)^2\lesssim_{\varepsilon_0,M} \alpha_{\sigma,2}(B_Q)^2 + \sum_{\substack{P\in\T^e\\P\subset Q}}\lVert\Delta^{\sigma}_P g^e_0\rVert^2_{L^2(\sigma)}\frac{\ell(P)}{\ell(Q)^{n+1}} + \sum_{\substack{S\in \Se\\S\subset Q}}\frac{\ell(S)^{2}}{\ell(Q)^{n+2}}\nu^e(S).
	\end{equation}
	Moreover,
	\begin{equation}\label{eq:sum of martingales yet again}
		\sum_{P\in\T^e}\lVert \Delta^{\sigma}_P g^e_0\rVert ^2_{L^2(\sigma)}\lesssim M\lVert g^e_0\rVert_{L^1(\sigma)}= M \nu^e_0(\Gamma)\le M\mu(\R^d).
	\end{equation}
	\end{lemma}
\begin{proof}	
	We claim that for $P\in\T^e$ with $\ell(P)\le 2^{-k}$ (in particular, for $P\in\T^e$ such that $P\subset Q$) we have
	\begin{equation}\label{eqn:deltas are equal}
	\Delta^{\sigma}_P g^e_k = \Delta^{\sigma}_P g^e_0.
	\end{equation}
	Indeed, for $x\not\in P$ both sides of \eqref{eqn:deltas are equal} are zero. For $x\in P'\subset P$, where $P'\in\T^e\cup\Se$ is a child of $P$, we have
	\begin{multline*}
	\Delta^{\sigma}_P g^e_0(x) - \Delta^{\sigma}_P g^e_k(x)=  \frac{\nu_0^e(P')-\nu^e_k(P')}{\ell(P')^n} -\frac{\nu_0^e(P)-\nu^e_k(P)}{\ell(P)^n}\\
	= \ell(P')^{-n}\left(\sum_{S\in\W^e_0\setminus\W_k^e}\frac{\mu(S)}{\ell(S)^n}\sg(P'\cap\Pi_{\Gamma}(S))\right)
	- \ell(P)^{-n}\left(\sum_{S\in\W^e_0\setminus\W_k^e}\frac{\mu(S)}{\ell(S)^n}\sg(P\cap\Pi_{\Gamma}(S))\right).
	\end{multline*}
	The Whitney cubes $S$ in the sums above above satisfy $\ell(S)>2^{-k}\ge \ell(P)$, and moreover we have $\Pi_{\Gamma}(S)\in\DGe$. Hence, we either have $P\cap\Pi_{\Gamma}(S)=P$ or $P\cap\Pi_{\Gamma}(S)=\varnothing $. The same is true for $P'$. Moreover, we have $P\cap\Pi_{\Gamma}(S)\neq\varnothing$ if and only if $P'\cap\Pi_{\Gamma}(S)\neq\varnothing$. It follows that the right hand side above is equal to
	\begin{equation*}
	\sum_{\substack{S\in\W^e_0\setminus\W_k^e\\P'\cap\Pi_{\Gamma}(S)\not=\varnothing}}\frac{\mu(S)}{\ell(S)^n}\ - \sum_{\substack{S\in\W^e_0\setminus\W_k^e\\P\cap\Pi_{\Gamma}(S)\not=\varnothing}}\frac{\mu(S)}{\ell(S)^n} = 0.
	\end{equation*}
	Thus $\Delta^{\sigma}_P g^e_k = \Delta^{\sigma}_P g^e_0.$ Using this equality, and also the fact that $\nu^e_k\le \nu^e$, we transform \eqref{eq:approximating estimate 4} into
	\begin{equation}\label{eq:yet another eqn}
	\widetilde{\alpha}_{{\nu}^e_k,2}(Q)^2\lesssim_{\varepsilon_0,M} \alpha_{\sigma,2}(B_Q)^2 + \sum_{\substack{P\in\T^e\\P\subset Q}}\lVert\Delta^{\sigma}_P g^e_0\rVert^2_{L^2(\sigma)}\frac{\ell(P)}{\ell(Q)^{n+1}} + \sum_{\substack{P\in \Se\\P\subset Q}}\frac{\ell(P)^{2}}{\ell(Q)^{n+2}}\nu^e(P).
	\end{equation}
	Concerning \eqref{eq:sum of martingales yet again}, it is an immediate consequence of \eqref{eq:sum of martingales on tree unbdd density} when we apply \lemref{lem:estimate alpha tilde unbdd density} to $\nu^e_0$ and the trees $\{Q\in\Tree^e\ :\ Q\subset R^e_i \}$ (recall that the union of such trees gives the entire $\Tree^e$).
\end{proof}

	We finally define $\T$ as the collection of cubes $Q\in\DG$ such that for every $e\in\{0,1\}^n$ there exists $P\in\T^e$ satsfying $\ell(P)=\ell(Q)$ and $P\cap Q\not = \varnothing $. It is easy to check that $\T$ is indeed a tree, and that the stopping cubes $\Stop=\Stop(\Tree)$ satisfy $\bigcup_{Q\in\Stop}Q\subset\bigcup_e\bigcup_{Q\in\Se}Q$. Thus,
	\begin{equation*}
	\mu\bigg(\bigcup_{Q\in\Stop}Q\bigg)\le \sum_{e\in\{0,1\}^n}\mu\bigg(\bigcup_{Q\in\Se}Q\bigg)\overset{\eqref{eq:Stope small}}{\le}\varepsilon.
	\end{equation*}
	Moreover, $\T\subset\T_{(0,\dots,0)}$, so for all $Q\in\T$ 
	\begin{align*} 
	\mu(\lambda\widetilde{B}_Q)&\le M\ell(Q)^n,\\ 
	\mu(Q)&\ge M^{-1}\ell(Q)^n.
	\end{align*}
	The only thing that remains to be shown is the packing condition \eqref{eq:packing condition on tree}.
	\begin{lemma}
		We have 
		\begin{equation*}
		\sum_{\substack{Q\in \T}} \widehat{\alpha}_{\nu_Q,2}(\widetilde{B}_Q)^2 \ell(Q)^n <\infty.
		\end{equation*}
	\end{lemma}
	\begin{proof}
	Recall that in \lemref{lem:Qtilde_corresp_to_Q} we defined a constant $k_0>0$ such that for any $Q\in\D_{\Gamma},\ \ell(Q)\le 2^{-k_0}$, there exists a cube $P_Q\in\EDG$ satisfying $3\widetilde{B}_Q\subset V(P_Q),\ \ell(P_Q)=2^{k_0}\ell(Q)$. Since there are only finitely many $Q\in\Tree$ with $\ell(Q)> 2^{-k_0}$, we may ignore them in the estimates that follow.
		
	Suppose $Q\in\T$ and $\ell(Q)\le 2^{-k_0}$, let $P_Q$ be as above. Recall that $\nu_Q=\nu^{e(Q)}_{k(Q)}$, where $e=e(Q),\ k=k(Q)$ are such that $P_Q\in\DGe$ and $\ell(P_Q)=2^{-k}$. 
	
	We defined $\T$ in such a way that necessarily $P_Q\in\T^e$. It follows from \lemref{lem:alphas on cubes} applied with $\nu=\nu_Q,\ B=\widetilde{B}_Q,\ Q=P_Q,$ that
	\begin{equation*}
	\widehat{\alpha}_{\nu_Q,2}(\widetilde{B}_Q)\lesssim_{\varepsilon_0,M,k_0}\widetilde{\alpha}_{\nu_Q,2}(P_Q) + \alpha_{\sigma,2}(B_{P_Q}).
	\end{equation*}
	We use \eqref{eq:final estimate of alpha tilde} and the inequality above to obtain
	\begin{equation*} 
	\widehat{\alpha}_{\nu_Q,2}(\widetilde{B}_Q)^2\lesssim_{\varepsilon_0,M,k_0}\alpha_{\sigma,2}(B_{P_Q})^2+ \sum_{\substack{P\in\T^e\\P\subset P_Q}}\lVert\Delta^{\sigma}_P g^e_0\rVert^2_{L^2(\sigma)}\frac{\ell(P)}{\ell(P_Q)^{n+1}} + \sum_{\substack{S\in \Se\\S\subset P_Q}}\frac{\ell(S)^{2}}{\ell(P_Q)^{n+2}}\nu^e(S).
	\end{equation*}
	Taking into account that each $P_Q\in\T^e$ may correspond to only a bounded number of $Q\in\T$, and that $\ell(Q)\approx_{k_0}\ell(P_Q)$, we get
	\begin{multline*}
	\sum_{\substack{Q\in\T: P_Q\in\T^e}}	\widehat{\alpha}_{\nu_Q,2}(\widetilde{B}_Q)^2\ell(Q)^n \lesssim_{\varepsilon_0,M,k_0}\sum_{Q'\in\T^e}\alpha_{\sigma,2}(B_{Q'})^2\ell(Q')^n\\ +\sum_{Q'\in\T^e}\sum_{\substack{P\in\T^e\\P\subset Q'}}\lVert\Delta^{\sigma}_P g^e_0\rVert^2_{L^2(\sigma)}\frac{\ell(P)}{\ell(Q')}
	+ \sum_{Q'\in\T^e}\sum_{\substack{S\in \Se\\S\subset Q'}}\frac{\ell(S)^{2}}{\ell(Q')^{2}}\nu^e(S).
	\end{multline*}
	The first sum from the right hand side is finite because $\sigma$ is uniformly rectifiable, see \thmref{thm:characterization UR}. We estimate the second sum by changing the order of summation:
	\begin{multline*}
	\sum_{Q'\in\T^e}\sum_{\substack{P\in\T^e\\P\subset Q'}}\lVert\Delta^{\sigma}_P g^e_0\rVert^2_{L^2(\sigma)}\frac{\ell(P)}{\ell(Q')} = \sum_{\substack{P\in\T^e}} \lVert\Delta^{\sigma}_P g^e_0\rVert^2_{L^2(\sigma)} \sum_{\substack{Q'\in\T^e\\Q'\supset P}}\frac{\ell(P)}{\ell(Q')}\\
	\lesssim \sum_{\substack{P\in\T^e}} \lVert\Delta^{\sigma}_P g^e_0\rVert^2_{L^2(\sigma)} \overset{\eqref{eq:sum of martingales yet again}}{\lesssim} M\mu(\R^d)<\infty.
	\end{multline*}
%	Recall that $g^e_0$ is the density of $\tilde{\nu}^e_0$. Since $\lVert\tilde{\nu}^e_0\rVert\le\lVert\mu\rVert<\infty$, we have $g^e_0\in L^1(\sigma)$. On the other hand, $g^e_0\in L^{\infty}(\sigma)$ by \lemref{lem:density bounded}. It follows that $\lVert g^e_0\rVert^2_{L^2(\sigma)}$ is finite.
	
	The third sum is treated similarly:
	\begin{equation*}
	\sum_{Q'\in\T^e}\sum_{\substack{S\in \Se\\S\subset Q'}}\frac{\ell(S)^{2}}{\ell(Q')^{2}}\nu^e(S) = \sum_{S\in\Se}\nu^e(S)\sum_{\substack{Q'\in \T^e\\Q'\supset S}}\frac{\ell(S)^{2}}{\ell(Q')^{2}}\lesssim \sum_{S\in\Se}\nu^e(S)<\infty. 
	\end{equation*}
	Thus,
	\begin{equation*}
	\sum_{\substack{Q\in\T}}\widehat{\alpha}_{\nu_Q,2}(\widetilde{B}_Q)^2\ell(Q)^n = \sum_{e\in\{0,1\}^n}\sum_{\substack{Q\in\T: P_Q\in\T^e}}	\widehat{\alpha}_{\nu_Q,2}(\widetilde{B}_Q)^2\ell(Q)^n<\infty.
	\end{equation*}
	\end{proof}
%\end{proof}
\section{From Approximating Measures to \texorpdfstring{$\mu$}{mu}}\label{sec:pass from approximating to mu}
%Going back from  \texorpdfstring{$\nu_Q$}{nu\_Q} to \texorpdfstring{$\mu$}{mu}
To prove \lemref{lem:general_with_Gamma} we need to pass from the estimates on $\widehat{\alpha}_{\nu_Q,2}(\widetilde{B}_Q)$ shown in \lemref{lem:T_lemma_general} to estimates on $\widehat{\alpha}_{\mu,2}(B_Q)$. 

Recall that $K>20$ is the constant such that for all Whitney cubes $Q\in\W^e$ we have $KQ\cap\Gamma\ne\varnothing$, and $k_0=k_0(n,\Lambda)$ is an integer from \lemref{lem:Qtilde_corresp_to_Q}.
\begin{lemma}\label{lem:estimate alpha mu with alpha sigma}
	There exists $\lambda=\lambda(k_0,K,n,d)>3$ such that if $M=M(\varepsilon,\lambda,\Lambda,n,d,\mu)$ and $\Tree=\Tree(\lambda,M,\varepsilon)$ are as in \lemref{lem:T_lemma_general}, then for all $Q\in\Tree$ with $\ell(Q)\le 2^{-k_0}$
	\begin{equation*}
	\widehat{\alpha}_{\mu,2}(B_Q)^2\lesssim_{M,\lambda,\Lambda}\widehat{\alpha}_{\nu_Q,2}(\widetilde{B}_Q)^2 + \alpha_{\sigma,2}(\widetilde{B}_Q)^2 + \frac{1}{\ell(Q)^{n+2}}\sum_{\substack{P\in \W_Q\\ P\subset \lambda\tilde{B}_Q}}\mu(P)\ell(P)^2.
	\end{equation*}
\end{lemma}
\begin{proof}
	Let $Q\in \T$ with $\ell(Q)\le 2^{-k_0}$. We will define an auxiliary measure $\mu_Q$. 
	Set 
	\begin{equation*} 
	I_Q = \{P\in\W_Q : \Pi_{\Gamma}(P)\cap 3\widetilde {B}_Q \ne \varnothing\}.
	\end{equation*}
	It is easy to check that
	\begin{equation}\label{eq:support of muQ}
	\bigcup_{P\in I_Q}P\subset \lambda\widetilde{B}_Q,
	\end{equation}
	for $\lambda=\lambda(k_0,K,n,d)$ big enough (e.g. $\lambda=C(n,d)K2^{k_0}$ works). It is crucial that all cubes in $I_Q$ have sidelength bounded by $2^{k_0}\ell(Q)$, otherwise no such $\lambda$ would exist.
	
	Recall that the functions $g_P(x) = \frac{\mu(P)}{\ell(P)^n}\one_{\Pi_{\Gamma}(P)}(x),\ P\in \W_Q,$ were used to define $\nu_Q$ at the beginning of Section \ref{sec:approximating measures nuQ}. Let
	\begin{equation*}
	a_P = \frac{\int \varphi_{\widetilde{B}_Q}g_P\ d\sigma}{\mu(P)}.
	\end{equation*}
	Note that for $P\in \W_Q\setminus I_Q$ we have $a_P=0$. The measure $\mu_Q$ is defined as
	\begin{equation*}
	\mu_Q = \varphi_{\widetilde{B}_Q}\restr{\mu}{\Gamma} + \sum_{P\in I_Q} a_P \restr{\mu}{P}.
	\end{equation*}
	
	First, let us show that if $\Lambda$ (the constant from the definition of $\widetilde{B}_Q=\Lambda B_Q$) is big enough, then $\restr{\mu}{3B_Q} = \restr{\mu_Q}{3B_Q}$. We need to check the following: if $P\in\W^{e(Q)}$ is such that $P\cap 3B_{Q}\ne \varnothing$, then $P\in I_Q$ and $a_P=1$.
	
	Note that for all such $P$ we have
	\begin{equation*}
	\ell(P)\le \diam(P)\overset{\eqref{eq:property of Whitney cubes}}{\le} r(3B_Q)= 9\diam(Q) \overset{\eqref{eq:k0 big}}{\le} 2^{-k(Q)},
	\end{equation*}
	and so $P\in \W_Q$. Furthermore, the fact that $P\cap 3B_{Q}\ne \varnothing$ and \eqref{eq:property of Whitney cubes} imply that $P\subset 9B_Q$. Since $\Pi_{\Gamma}$ is $\sqrt{2}$-Lipschitz continuous, and $B_Q$ is centered at $\Gamma$, we get that for $\Lambda$ big enough (e.g. $\Lambda=9\sqrt{2}$)
	\begin{equation}\label{eq:fixing Lambda}
	\Pi_{\Gamma}(P)\subset \Lambda B_Q = \widetilde{B}_Q.
	\end{equation}
	We conclude that $P\in I_Q$ and $a_P=1$, and so, 
	\begin{equation}\label{eq:mu equal to mu_Q}
	\restr{\mu}{3B_Q} = \restr{\mu_Q}{3B_Q}.
	\end{equation}
	
	Set $L = L_{\widetilde{B}_Q}$. We will apply \lemref{lem:W_2 est Tolsa} with $\nu=\mu_Q,\ B_1=B_Q,\ B_2=\lambda\widetilde{B}_Q,\ L=L,$ and $f=\varphi_{\widetilde{B}_Q}$. Notice that $\supp\mu_Q\subset \lambda\widetilde{B}_Q$ by \eqref{eq:support of muQ}. Moreover, using the same trick as in the beginning of the proof of \lemref{lem:alpha_on_balls}, we may assume that $L\cap B_Q\neq\varnothing$. Since $\mu_Q(B_Q)\approx_M\mu_Q(\lambda\widetilde{B}_Q)\approx_M \ell(Q)^n$ by \lemref{lem:T_lemma_general}, and $r(\lambda \widetilde{B}_Q)=\lambda\Lambda r(B_Q)$, the assumptions of \lemref{lem:W_2 est Tolsa} are met, and we get that
	\begin{gather}\label{eq:estimating alpha hat}
	W_2(\varphi_{Q}\mu_Q, a\varphi_{Q}\Hn{L})\lesssim_{M,\lambda,\Lambda} W_2(\mu_Q, a\varphi_{\widetilde{B}_Q}\Hn{L}).
	\end{gather}
	
	%		For brevity of notation set $L = L_{\widetilde{B}_Q}.$ 
	Applying the triangle inequality yields
	\begin{multline}\label{eqn:W_2^estimates_last_lemma}
	W_2(\mu_Q,a\varphi_{\widetilde{B}_Q}\Hn{L})^2
	\lesssim W_2(\mu_Q,\varphi_{\widetilde{B}_Q}\nu_Q)^2 + 
	W_2(\varphi_{\widetilde{B}_Q}\nu_Q,a\varphi_{\widetilde{B}_Q}\Hn{L})^2\\
	\approx_M  W_2(\mu_Q,\varphi_{\widetilde{B}_Q}\nu_Q)^2 + \widehat{\alpha}_{\nu_Q,2}(\widetilde{B}_Q)^2\ell(Q)^{n+2}.
	\end{multline}
	
	To estimate $W_2(\mu_Q,\varphi_{\widetilde{B}_Q}\nu_Q)$ we define the following transport plan:
	\begin{equation*}
	d\pi(x,y) = \varphi_{\widetilde{B}_Q}(x)d\restr{\mu}{\Gamma}(x)d\delta_{x}(y) + \sum_{P\in I_Q}\frac{1}{\mu_Q(P)}d\restr{\mu_Q}{P}(x) \varphi_{\widetilde{B}_Q}(y)g_P(y)d\sigma(y).
	\end{equation*}
	Then,
	\begin{multline*}
	W_2(\mu_Q,\varphi_{\widetilde{B}_Q}\nu_Q)^2\le \int |x-y|^2\ d\pi(x,y) \lesssim \sum_{P\in I_Q}\ell(P)^2\int \varphi_{\widetilde{B}_Q}(y)g_P(y)d\sigma(y).\\
	\le \sum_{P\in I_Q}\mu(P)\ell(P)^2\overset{\eqref{eq:support of muQ}}{\le}\sum_{\substack{P\in \W_Q\\ P\subset \lambda\widetilde{B}_Q}}\mu(P)\ell(P)^2.
	\end{multline*}
	Putting together \eqref{eq:mu equal to mu_Q}, \eqref{eq:estimating alpha hat}, \eqref{eqn:W_2^estimates_last_lemma}, and the estimate above, we get
	\begin{equation*}
	W_2(\varphi_{Q}\mu, a\varphi_{Q}\Hn{L})\lesssim_{M,\lambda,\Lambda} \widehat{\alpha}_{\nu_Q,2}(\widetilde{B}_Q)^2\ell(Q)^{n+2} + \sum_{\substack{P\in \W_Q\\ P\subset \lambda\widetilde{B}_Q}}\mu(P)\ell(P)^2.
	\end{equation*}
	Finally, we use the triangle inequality, the estimate $\mu(3B_Q)\approx_{M} \sigma(B_Q)\approx r(B_Q)^n$, and the fact that $L_Q$ minimizes $\alpha_{\sigma,2}(B_Q)$, to get
	\begin{multline*}
	\widehat{\alpha}_{\mu,2}(B_Q)^2\ell(Q)^{n+2}\approx_M W_2(\varphi_{Q}\mu, a\varphi_{Q}\Hn{L_Q})\leq W_2(\varphi_Q\mu,a\varphi_Q\Hn{L}) \\
	+ \left(\frac{\int \varphi_Q\ d\mu}{\int \varphi_Q\ d\sigma}\right)^{1/2}\left(W_2(\varphi_Q\sigma,a\varphi_Q\Hn{L_Q}) +
	W_2(\varphi_Q\sigma,a\varphi_Q\Hn{L})\right)\\
	\lesssim_M W_2(\varphi_{Q}\mu, a\varphi_{Q}\Hn{L}) + W_2(\varphi_{Q}\sigma, a\varphi_{Q}\Hn{L})\\
	\lesssim W_2(\varphi_{Q}\mu, a\varphi_{Q}\Hn{L}) + \alpha_{\sigma,2}(\widetilde{B}_Q)^2\ell(Q)^{n+2},
	\end{multline*}
	and so the proof is complete.
\end{proof}

We are ready to finish the proof of \lemref{lem:general_with_Gamma}.

\begin{proof}[Proof of \lemref{lem:general_with_Gamma}]
	Recall that $R$ is a $\Gamma$-cube with $\ell(R)=1$, and $\varepsilon>0$ is an arbitrary small constant, and that they were both fixed in Subsection \ref{subsec:notation}. Let $\lambda,\ M,\ \Tree,$ and $\Stop$ be as in \lemref{lem:estimate alpha mu with alpha sigma} and \lemref{lem:T_lemma_general}. Set
	\begin{equation*}
	R' = R\setminus\bigcup_{P\in\Stop}P.
	\end{equation*}
	By \lemref{lem:T_lemma_general}, we have $\mu(R')\ge (1-\varepsilon)\mu(R)$. Our aim is to show that
	\begin{equation*}
	\int_{R'}\int_0^{1}\alpha_{\mu,2}(x,r)^2\ \frac{dr}{r}\ d\mu(x)<\infty.
	\end{equation*}
	
	For any $x\in R'$ we have arbitrarily small cubes from $\Tree$ containing $x$. Hence, for any $k\ge k_0+3,\ r\in (2^{-k},2^{-k+1}],$ we have $3B(x,r)\subset B_Q$ for the cube $Q\in\T$ containing $x$ and satisfying $\ell(Q)=2^{-k+3}$. Thus, by \lemref{lem:alpha_on_balls},
	\begin{equation*}
	\widehat{\alpha}_{\mu,2}(B(x,r))^2\lesssim_M \widehat{\alpha}_{\mu,2}(B_Q)^2 + \alpha_{\sigma,2}(B_Q)^2.
	\end{equation*}
	Integrating both sides with respect to $r$ yields
	\begin{equation*}
	\int_{2^{-k}}^{2^{-k+1}} \widehat{\alpha}_{\mu,2}(B(x,r))^2\ \frac{dr}{r} \lesssim_M \int_{2^{-k}}^{2^{-k+1}}( \widehat{\alpha}_{\mu,2}(B_Q)^2 + \alpha_{\sigma,2}(B_Q)^2)\ \frac{dr}{r} \approx \widehat{\alpha}_{\mu,2}(B_Q)^2 + \alpha_{\sigma,2}(B_Q)^2.
	\end{equation*}
	The inequality above holds for all $x\in Q\cap R'$, so
	\begin{multline*} 
	\int_{Q\cap R'} \int_{2^{-k}}^{2^{-k+1}} \widehat{\alpha}_{\mu,2}(B(x,r))^2\ \frac{dr}{r}\ d\mu(x) \lesssim_M (\widehat{\alpha}_{\mu,2}(B_Q)^2 + \alpha_{\sigma,2}(B_Q)^2)\mu(Q)\\
	\approx_M (\widehat{\alpha}_{\mu,2}(B_Q)^2 + \alpha_{\sigma,2}(B_Q)^2)\ell(Q)^n.
	\end{multline*}
	Summing over all $Q\in\T$ with $\ell(Q)=2^{-k+3}$, and then over all $k\ge k_0+3$, we get
	\begin{equation}\label{eqn:passing to sums}
	\int_{R'}\int_0^{2^{-k_0-2}}\widehat{\alpha}_{\mu,2}(B(x,r))^2\ \frac{dr}{r}\ d\mu(x) \lesssim_M \sum_{\substack{Q\in\T\\ \ell(Q)\le 2^{-k_0}}} \widehat{\alpha}_{\mu,2}(B_Q)^2\ell(Q)^n + 
	\sum_{\substack{Q\in\T\\ \ell(Q)\le 2^{-k_0}}}\alpha_{\sigma,2}(B_Q)^2\ell(Q)^n.
	\end{equation}
	On the other hand, for any $r>0$ we have 
	\begin{equation*}
	\widehat{\alpha}_{\mu,2}(B(x,r))^2\lesssim \frac{\mu(\R^d)}{r^n},
	\end{equation*}
	so
	\begin{equation*}
	\int_{R'}\int_{2^{-k_0-2}}^{1}\widehat{\alpha}_{\mu,2}(B(x,r))^2\ \frac{dr}{r}\ d\mu(x) < \infty.
	\end{equation*}
	Thus, in order to prove \lemref{lem:general_with_Gamma}, it suffices to show that the sums on the right hand side of \eqref{eqn:passing to sums} are finite. 
	
	The finiteness of
	\begin{equation*} 
	\sum_{Q\in\D_{\Gamma},\ Q\subset R}\alpha_{\sigma,2}(B_Q)^2\ell(Q)^n
	\end{equation*}
	follows by \thmref{thm:characterization UR}.
	To estimate the other sum we apply \lemref{lem:estimate alpha mu with alpha sigma}:
	\begin{multline*}
	\sum_{\substack{Q\in\T\\ \ell(Q)\le 2^{-k_0}}} \widehat{\alpha}_{\mu,2}(B_Q)^2\ell(Q)^n \lesssim \sum_{\substack{Q\in\T\\ \ell(Q)\le 2^{-k_0}}}\widehat{\alpha}_{\nu_Q,2}(\widetilde{B}_Q)^2\ell(Q)^n + \sum_{\substack{Q\in\T\\ \ell(Q)\le 2^{-k_0}}}\alpha_{\sigma,2}(\widetilde{B}_Q)^2\ell(Q)^n\\
	 + \sum_{\substack{Q\in\T\\ \ell(Q)\le 2^{-k_0}}}\sum_{\substack{P\in \W_Q\\ P\subset \lambda\tilde{B}_Q}}\mu(P)\frac{\ell(P)^2}{\ell(Q)^2}.
	\end{multline*}
	The first sum is finite by \lemref{lem:T_lemma_general}, the second by \thmref{thm:characterization UR}. Concerning the last sum, we may estimate it in the following way:
	\begin{multline*}
	\sum_{\substack{Q\in\T\\ \ell(Q)\le 2^{-k_0}}}\sum_{\substack{P\in \W_Q\\ P\subset \lambda\tilde{B}_Q}}\mu(P)\frac{\ell(P)^2}{\ell(Q)^2} \lesssim \sum_{e\in\{0,1\}^n}\sum_{\substack{P\in \W^e\\ P\subset \lambda\tilde{B}_R}} \mu(P) \sum_{\substack{Q\in\T\\ \lambda \tilde{B}_{Q}\supset P}} \frac{\ell(P)^2}{\ell(Q)^2}\\
	\lesssim \sum_{e\in\{0,1\}^n}\sum_{\substack{P\in \W^e\\ P\subset \lambda\tilde{B}_R}} \mu(P) \le \sum_{e\in\{0,1\}^n} \mu(\lambda \widetilde{B}_R) = 2^n \mu(\lambda \widetilde{B}_R)<\infty.
	\end{multline*}
	Thus,
	\begin{equation*}
	\sum_{Q\in\T} \widehat{\alpha}_{\mu,2}(B_Q)^2\ell(Q)^n<\infty.
	\end{equation*}
\end{proof}


\begin{thebibliography}{ADT16}
	\expandafter\ifx\csname url\endcsname\relax
	\def\url#1{\texttt{#1}}\fi
	\expandafter\ifx\csname doi\endcsname\relax
	\def\doi#1{\burlalt{doi:#1}{http://dx.doi.org/#1}}\fi
	\expandafter\ifx\csname urlprefix\endcsname\relax\def\urlprefix{URL }\fi
	\expandafter\ifx\csname href\endcsname\relax
	\def\href#1#2{#2}\fi
	\expandafter\ifx\csname burlalt\endcsname\relax
	\def\burlalt#1#2{\href{#2}{#1}}\fi
	
	\bibitem[ADT16]{azzam2016wasserstein}
	J.~Azzam, G.~David, and T.~Toro.
	\newblock {Wasserstein} distance and the rectifiability of doubling measures:
	part {I}.
	\newblock {\em Math. Ann.}, 364(1-2):151--224, 2016,
	\burlalt{arXiv:1408.6645}{http://arxiv.org/abs/1408.6645}.
	\newblock \doi{10.1007/s00208-015-1206-z}.
	
	\bibitem[AM16]{azzam2016characterization}
	J.~Azzam, and M.~Mourgoglou.
	\newblock A characterization of 1-rectifiable doubling measures with connected
	supports.
	\newblock {\em Anal. PDE}, 9(1):99--109, 2016,
	\burlalt{arXiv:1501.02220}{http://arxiv.org/abs/1501.02220}.
	\newblock \doi{10.2140/apde.2016.9.99}.
	
	\bibitem[AT15]{azzam2015characterization}
	J.~Azzam, and X.~Tolsa.
	\newblock Characterization of $n$-rectifiability in terms of {Jones'} square
	function: Part {II}.
	\newblock {\em Geom. Funct. Anal.}, 25(5):1371--1412, 2015,
	\burlalt{arXiv:1501.01572}{http://arxiv.org/abs/1501.01572}.
	\newblock \doi{10.1007/s00039-015-0334-7}.
	
	\bibitem[ATT18]{azzam2018characterization}
	J.~{Azzam}, X.~{Tolsa}, and T.~{Toro}.
	\newblock {Characterization of rectifiable measures in terms of
		$\alpha$-numbers}.
	\newblock {\em Preprint}, 2018,
	\burlalt{arXiv:1808.07661}{http://arxiv.org/abs/1808.07661}.
	
	\bibitem[Bad19]{badger2018generalized}
	M.~Badger.
	\newblock Generalized rectifiability of measures and the identification
	problem.
	\newblock {\em Complex Anal. Synerg.}, 5(1):2, 2019,
	\burlalt{arXiv:1803.10022}{http://arxiv.org/abs/1803.10022}.
	\newblock \doi{10.1007/s40627-019-0027-3}.
	
	\bibitem[BS15]{badger2015multiscale}
	M.~Badger, and R.~Schul.
	\newblock Multiscale analysis of 1-rectifiable measures: necessary conditions.
	\newblock {\em Math. Ann.}, 361(3-4):1055--1072, 2015,
	\burlalt{arXiv:1307.0804}{http://arxiv.org/abs/1307.0804}.
	\newblock \doi{10.1007/s00208-014-1104-9}.
	
	\bibitem[BS16]{badger2016two}
	M.~Badger, and R.~Schul.
	\newblock Two sufficient conditions for rectifiable measures.
	\newblock {\em Proc. Amer. Math. Soc.}, 144(6):2445--2454, 2016,
	\burlalt{arXiv:1412.8357}{http://arxiv.org/abs/1412.8357}.
	\newblock \doi{10.1090/proc/12881}.
	
	\bibitem[BS17]{badger2017multiscale}
	M.~Badger, and R.~Schul.
	\newblock Multiscale analysis of 1-rectifiable measures {II}:
	Characterizations.
	\newblock {\em Anal. Geom. Metr. Spaces}, 5(1):1--39, 2017,
	\burlalt{arXiv:1602.03823}{http://arxiv.org/abs/1602.03823}.
	\newblock \doi{10.1515/agms-2017-0001}.
	
	\bibitem[D{\k{a}}b19]{dabrowski2019sufficient}
	D.~D{\k{a}}browski.
	\newblock Sufficient condition for rectifiability involving {Wasserstein}
	distance {$W_2$}.
	\newblock {\em Preprint}, 2019,
	\burlalt{arXiv:1904.11004}{http://arxiv.org/abs/1904.11004}.
	
	\bibitem[Dav91]{david1991wavelets}
	G.~David.
	\newblock {\em Wavelets and Singular Integrals on Curves and Surfaces}, volume
	1465 of {\em Lecture Notes in Math.}
	\newblock Springer, 1991.
	\newblock \doi{10.1007/BFb0091544}.
	
	\bibitem[DS91]{david1991singular}
	G.~David, and S.~Semmes.
	\newblock Singular integrals and rectifiable sets in {$\mathbb{R}^n$}:
	Au-del\`{a} des graphes lipschitziens.
	\newblock {\em Ast{\'e}risque}, 193, 1991.
	\newblock \doi{10.24033/ast.68}.
	
	\bibitem[DS93]{david1993analysis}
	G.~David, and S.~Semmes.
	\newblock {\em Analysis of and on Uniformly Rectifiable Sets}, volume~38 of
	{\em Math. Surveys Monogr.}
	\newblock Amer. Math. Soc., 1993.
	\newblock \doi{10.1090/surv/038}.
	
	\bibitem[ENV16]{edelen2016quantitative}
	N.~Edelen, A.~Naber, and D.~Valtorta.
	\newblock Quantitative {Reifenberg} theorem for measures.
	\newblock {\em Preprint}, 2016,
	\burlalt{arXiv:1612.08052}{http://arxiv.org/abs/1612.08052}.
	
	\bibitem[GKS10]{garnett2010doubling}
	J.~Garnett, R.~Killip, and R.~Schul.
	\newblock A doubling measure on $\mathbb{R}^d$ can charge a rectifiable curve.
	\newblock {\em Proc. Amer. Math. Soc.}, 138(5):1673--1679, 2010,
	\burlalt{arXiv:0906.2484}{http://arxiv.org/abs/0906.2484}.
	\newblock \doi{10.1090/S0002-9939-10-10234-2}.
	
	\bibitem[Gra08]{grafakos2008classical}
	L.~Grafakos.
	\newblock {\em Classical Fourier Analysis}, volume 249 of {\em Grad. Texts in
		Math.}
	\newblock Springer, 2008.
	\newblock \doi{10.1007/978-1-4939-1194-3}.
	
	\bibitem[Jon90]{jones1990rectifiable}
	P.~W. Jones.
	\newblock Rectifiable sets and the traveling salesman problem.
	\newblock {\em Invent. Math.}, 102(1):1--15, 1990.
	\newblock \doi{10.1007/BF01233418}.
	
	\bibitem[Ler03]{lerman2003quantifying}
	G.~Lerman.
	\newblock Quantifying curvelike structures of measures by using {$L^2$} {Jones}
	quantities.
	\newblock {\em Comm. Pure Appl. Math.}, 56(9):1294--1365, 2003.
	\newblock \doi{10.1002/cpa.10096}.
	
	\bibitem[Mat95]{mattila1999geometry}
	P.~Mattila.
	\newblock {\em Geometry of sets and measures in {Euclidean} spaces: fractals
		and rectifiability}, volume~44 of {\em Cambridge Stud. Adv. Math.}
	\newblock Cambridge Univ. Press, 1995.
	\newblock \doi{10.1017/CBO9780511623813}.
	
	\bibitem[MO18]{martikainen2018boundedness}
	H.~Martikainen, and T.~Orponen.
	\newblock Boundedness of the density normalised {Jones}' square function does
	not imply 1-rectifiability.
	\newblock {\em J. Math. Pures Appl.}, 110:71--92, 2018,
	\burlalt{arXiv:1604.04091}{http://arxiv.org/abs/1604.04091}.
	\newblock \doi{10.1016/j.matpur.2017.07.009}.
	
	\bibitem[Oki92]{okikiolu1992characterization}
	K.~Okikiolu.
	\newblock Characterization of subsets of rectifiable curves in
	{$\mathbb{R}^n$}.
	\newblock {\em J. Lond. Math. Soc. (2)}, 46(2):336--348, 1992.
	\newblock \doi{10.1112/jlms/s2-46.2.336}.
	
	\bibitem[Orp18]{orponen2018absolute}
	T.~Orponen.
	\newblock Absolute continuity and $\alpha$-numbers on the real line.
	\newblock {\em Anal. PDE}, 12(4):969--996, 2018,
	\burlalt{arXiv:1703.02935}{http://arxiv.org/abs/1703.02935}.
	\newblock \doi{10.2140/apde.2019.12.969}.
	
	\bibitem[Paj97]{pajot1997conditions}
	H.~Pajot.
	\newblock Conditions quantitatives de rectifiabilit\'e.
	\newblock {\em Bull. Soc. Math. France}, 125(1):15--53, 1997.
	\newblock \doi{10.24033/bsmf.2298}.
	
	\bibitem[Ste70]{stein1970singular}
	E.~M. Stein.
	\newblock {\em Singular Integrals and Differentiability Properties of
		Functions}, volume~30 of {\em Princeton Math. Ser.}
	\newblock Princeton Univ. Press, 1970.
	
	\bibitem[Tol09]{tolsa2008uniform}
	X.~Tolsa.
	\newblock Uniform rectifiability, {Calder{\'o}n}-{Zygmund} operators with odd
	kernel, and quasiorthogonality.
	\newblock {\em Proc. Lond. Math. Soc. (3)}, 98(2):393--426, 2009,
	\burlalt{arXiv:0805.1053}{http://arxiv.org/abs/0805.1053}.
	\newblock \doi{10.1112/plms/pdn035}.
	
	\bibitem[Tol12]{tolsa2012mass}
	X.~Tolsa.
	\newblock Mass transport and uniform rectifiability.
	\newblock {\em Geom. Funct. Anal.}, 22(2):478--527, 2012,
	\burlalt{arXiv:1103.1543}{http://arxiv.org/abs/1103.1543}.
	\newblock \doi{10.1007/s00039-012-0160-0}.
	
	\bibitem[Tol15]{tolsa2015characterization}
	X.~Tolsa.
	\newblock Characterization of $n$-rectifiability in terms of {Jones'} square
	function: part {I}.
	\newblock {\em Calc. Var. Partial Differential Equations}, 54(4):3643--3665,
	2015, \burlalt{arXiv:1501.01569}{http://arxiv.org/abs/1501.01569}.
	\newblock \doi{10.1007/s00526-015-0917-z}.
	
	\bibitem[Tol17]{tolsa2017rectifiable}
	X.~Tolsa.
	\newblock Rectifiable measures, square functions involving densities, and the
	{Cauchy} transform.
	\newblock {\em Mem. Amer. Math. Soc.}, 245(1158), 2017,
	\burlalt{arXiv:1408.6979}{http://arxiv.org/abs/1408.6979}.
	\newblock \doi{10.1090/memo/1158}.
	
	\bibitem[Tol19]{tolsa2017rectifiability}
	X.~Tolsa.
	\newblock Rectifiability of measures and the $\beta_p$ coefficients.
	\newblock {\em Publ. Mat.}, 63(2):491--519, 2019,
	\burlalt{arXiv:1708.02304}{http://arxiv.org/abs/1708.02304}.
	\newblock \doi{10.5565/PUBLMAT6321904}.
	
	\bibitem[TT15]{tolsa2014rectifiability}
	X.~Tolsa, and T.~Toro.
	\newblock Rectifiability via a square function and {Preiss}’ theorem.
	\newblock {\em Int. Math. Res. Not. IMRN}, 2015(13):4638--4662, 2015,
	\burlalt{arXiv:1402.2799}{http://arxiv.org/abs/1402.2799}.
	\newblock \doi{10.1093/imrn/rnu082}.
	
	\bibitem[Vil03]{villani2003topics}
	C.~Villani.
	\newblock {\em Topics in optimal transportation}, volume~58 of {\em Grad. Stud.
		Math.}
	\newblock Amer. Math. Soc., 2003.
	\newblock \doi{10.1090/gsm/058}.
	
	\bibitem[Vil08]{villani2008optimal}
	C.~Villani.
	\newblock {\em Optimal transport: old and new}, volume 338 of {\em Grundlehren
		Math. Wiss.}
	\newblock Springer, 2008.
	\newblock \doi{10.1007/978-3-540-71050-9}.
	
\end{thebibliography}
\end{document}